\RequirePackage[T1]{fontenc}
\documentclass[12pt,twoside]{amsart}

\usepackage{TMF}

\begin{document}

\title[Genuine $C_n$-equivariant $\TMF$]{Genuine $C_n$-equivariant $\TMF$}
\author{Ying-Hsuan Lin}
\author{Akira Tominaga}
\author{Mayuko Yamashita}
\date{}
\address{Jefferson Physical Laboratory, Harvard University, Cambridge, MA 02138, USA}
\email{yhlin@alum.mit.edu}
\address{Department of Mathematics, Johns Hopkins University, Baltimore, MD 21218, USA}
\email{atomina1@jh.edu}
\address{Perimeter Institute for Theoretical Physics / RIKEN, 31 Caroline Street North,
Waterloo, Ontario, Canada, N2L 2Y5}
\email{myamashita@perimeterinstitute.ca}
\thanks{{\it Acknowledgments}: 
	The authors thank Lennart Meier, Tilman Bauer, and David Gepner for helpful discussions.
	The work of MY at Perimeter Institute is supported in part by the Government of Canada through the Department of Innovation, Science and Economic Development and by the Province of Ontario through the Ministry of Colleges and Universities. She is also supported by Grant-in-Aid for JSPS KAKENHI Grant Number 20K14307 and JST CREST program JPMJCR18T6, as well as the Simons Collaboration on Global Categorical Symmetries (Simons Foundation International grant SFI-MPS-GCS-00008528). YL was supported by the Simons Collaboration Grant on the Non-Perturbative Bootstrap for a portion of this work. This research was partly supported by a grant NSF PHY-1748958 to the Kavli Institute for Theoretical Physics.
	}

\begin{abstract}
	We determine the $\mathrm{TMF}$-module structures of the genuine $C_2$-equivariant $\mathrm{TMF}$ with $\mathrm{RO}(C_2)$-gradings and of the $C_3$-equivariant $\mathrm{TMF}$. Moreover, we propose a general strategy for studying $C_n$-equivariant $\mathrm{TMF}$ via $U(1)$-equivariant $\mathrm{TMF}$ and a duality phenomenon in equivariant $\mathrm{TMF}$.
\end{abstract}

\maketitle

\tableofcontents

\section{Introduction}\label{sec_intro}

Elliptic cohomology and its equivariant refinements have been of central interest in algebraic topology, representation theory, and mathematical physics. Numerous pioneering works have explored equivariant elliptic cohomology, far too many to list exhaustively. Nevertheless, we highlight the foundational work of Lurie \cite{LurieElliptic3} and Gepner-Meier \cite{GepnerMeier} on formulating the equivariant refinement of the spectrum of Topological Modular Forms ($\TMF$). 

In this paper, we study the equivariant $\TMF$ for cyclic groups. Denote by $C_n$ the cyclic group of order $n$. 
We propose a general strategy for analyzing the $C_n$-equivariant $\TMF$ by reducing to the $U(1)$-equivariant case, which exhibits more tractable structural behavior. This strategy allows us to determine the structures of the $C_n$-equivariant $\TMF$ without resorting to a full computation of the descent spectral sequence.

While various $p$-local studies of $C_n$-equivariant $\TMF$ were known (e.g., \cite{meier2022topologicalmodularformslevel}, \cite{chua}), a unified picture of the integral structure had remained unclear until recent works including this paper.
Following \cite{GepnerMeier}, the $C_n$-equivariant $\TMF$ is defined as
\begin{align}
	\TMF^{C_n} := \Gamma(\cE^\ori[n]; \cO_{\cE^\ori[n]}),
\end{align} 
where $\cE^\ori[n]$ denotes the $n$-torsion points (the kernel of the $n$-fold multiplication map) of $\cE^{\ori} \to \cM^{\ori}$, the universal oriented curve in the sense of spectral algebraic geometry \cite{LurieElliptic2}.
Meier \cite{meier2022topologicalmodularformslevel} established the additive decomposition of $\TMF^{C_n}$ after $p$-completion, where $p$ is a prime. In particular, when $p$ does not divide $n$, the $p$-localized $\TMF_{(p)}^{C_n}$ splits as a direct sum of shifts of $\TMF_1(3)$, $\TMF_1(2)$, and $\TMF$. Moreover, $\TMF_1(3)$ and $\TMF_1(2)$ can each be described as the smash product of $\TMF$ with a finite cell complex \cite{Mathew}. 

In addition, Chua \cite{chua} computed the descent spectral sequence of $2$-local $\TMF^{C_2}$ and showed that $\TMF^{C_2}$ can likewise be expressed as the smash product of $\TMF$ with a finite cell complex (see \cite{chua} and Corollary \ref{CorChua}).
However, Chua also noted the difficulty of computing the descent spectral sequence in the $3$-local $C_3$-equivariant case, owing to the complexity of the multiplication-by-$3$ formula for elliptic curves. In general, since $\TMF^{C_n} = \Gamma(\cE^{\ori}[n], \cO_{\cE^\ori [n]})$ is a $\TMF$-module of rank $n^2$, the descent spectral sequence is inherently complicated for larger $n$.

In contrast to the cyclic-group case, $U(1)$-equivariant $\TMF$ exhibits a much simpler structure. In \cite{GepnerMeier}, Gepner and Meier defined the genuine $U(1)$-fixed point spectrum as the global section of the structure sheaf of the universal oriented elliptic curve
\begin{align}
	\TMF^{U(1)} \coloneqq \Gamma(\cE^\ori; \cO_{\cE^\ori})  
\end{align}
and established the additive decomposition
\begin{align}
    \TMF^{U(1)} \simeq \TMF \oplus \Sigma \TMF.
\end{align}
More generally, $\RO(U(1))$-graded $\TMF$ has been studied by the second author \cite{Tominaga} and by Bauer-Meier \cite{BauerMeierTJF}; see also the ``user's guide'' in \cite[Appendix A]{lin2024topologicalellipticgenerai}. Named \textit{Topological Jacobi Forms}, it represents a spectral refinement of the ring of integral Jacobi Forms (see Section~\ref{subsec_TJF} for further explanation). 

We employ the $\RO(U(1))$-graded, $U(1)$-equivariant $\TMF$ to analyze $\TMF^{C_n}$.
Specifically, the group extension
\begin{align}
	C_n \hookrightarrow U(1) \xrightarrow{(-)^n} U(1)
\end{align}
induces the following fiber sequence of $\TMF$-module spectra:
\begin{align}\label{seq_intro}
	\xymatrix{
	\Gamma(\cE^\ori, p^*\omega) \ar[r] \ar@{=}[d] & \Gamma(\cE^\ori, \cO_{\cE^\ori}(n^2e)\otimes p^*\omega^{n^2}) \ar[r]^-{\res_{\cE[n]}}  \ar@{=}[d] & \Gamma(\cE^\ori[n], \cO_{\cE[n]}) \ar@{=}[d]  \\
	\Sigma^{-2}\TMF^{U(1)} \ar[r] & \Sigma^{-2} ( \TMF \otimes S^{\sigma^n})^{U(1)} \ar[r]^-{\res_{U(1)}^{C_n}}& \TMF^{C_n}
 }
\end{align}
Here, $\cO_\cE(n^2 e)$ denotes the sheaf of meromorphic functions with poles of order at most $n^2$ and located only on the zero section $e$ of the universal elliptic curve. We further generalize this fiber sequence to the $\RO(C_n)$-graded setting. These results are developed in Section \ref{sec_general} and provide efficient tools for the study of $C_n$-equivariant $\TMF$. 

In Sections \ref{sec_app1} and \ref{sec_C3}, we apply this general strategy to the cases $n=2$ and $n=3$. For $n=2$, the fiber sequence \eqref{seq_intro} is computable, allowing us to recover the results of \cite{meier2022topologicalmodularformslevel} and \cite{chua} for $\pi_* \TMF^{C_2}$ without localizing at any prime. Similarly, Theorem \ref{thm_C3_3loc_main} determines the $\TMF$-module structure of the 3-local fixed points $\TMF^{C_3}$. As the homotopy groups of $\TMF^{C_3}$ had not been computed before, our work completes the calculation of the cyclic-group-equivariant $\TMF$ of prime order.

We also compute the $\RO(C_2)$-graded equivariant $\TMF$ using the above fiber sequence, together with results from \cite{lin2024topologicalellipticgenerai}. The twists of $C_2$-equivariant $\TMF$ are classified by $[BC_2, P^4 BO] \simeq \Z/8$ (see \cite{Lurie2009} and \cite{ando2010twistsktheorytmf}), and we determine the $\TMF$-module structure of each twisted $\TMF$. The main results appear in Theorems~\ref{thm_nontwist_C2} and~\ref{thm_C2_twisted} of Section~\ref{sec_app1}, where we identify an elegant pattern of cell diagrams (see Figure \ref{celldiag_TMFZ/2}). 
The results verify the {\it level-rank dualities} between $C_2 = O(1)$-equivariant and $\Spin(k)$-equivariant $\TMF$ (see Section~\ref{subsec_notations} for our conventions on $\RO(G)$-gradings and dualities):
\ie
	\TMF[k\lambda]^{C_2} \simeq D \left(\TMF[\overline V_{\Spin(k)}]^{\Spin(k)} \right) \mbox{ for } 2 \le k \le 6, 
\fe
mirroring the dualities known in the context of modular tensor categories. These equivalences are regarded as variants of the level-rank dualities for $U/SU$ and $Sp/Sp$, verified in \cite{lin2024topologicalellipticgenerai}. 

This paper is organized as follows. Section \ref{sec_preliminaries} reviews necessary preliminaries. In Section \ref{sec_general}, we introduce the general strategy and set up the fiber sequence \eqref{seq_intro}. 
Sections \ref{sec_app1} and \ref{sec_C3} apply this strategy to determine the structures of $C_2$- and $C_3$-equivariant $\TMF$, respectively. 
Appendix \ref{appendix_TJF2inverted} explains the $\TMF$-module structure of the $3$-local Topological Jacobi Forms used in Section \ref{sec_C3}, while Appendix \ref{appendix_TJF2} discusses the $2$-local case. The full analysis of the $2$-local $\TJF$ is in \cite{Tominaga}; Appendix \ref{appendix_TJF2} specifically highlights some properties of $\TJF$ from the spectral sequence computation of $\pi_* \TJF$. 

\subsection{Notations and conventions}\label{subsec_notations}

\begin{itemize}
    \item $\cM^{\ori}$ denotes the spectral Deligne-Mumford stack of oriented elliptic curves, and $\cE^{\ori} \to \cM^{\ori}$ denotes the universal oriented curve. In particular, $\TMF$, the spectrum of Topological Modular Forms, is the global section of the structure sheaf, $\TMF = \Gamma(\cM^{\ori}, \cO_{\cM^{\ori}} ).$
    \item For a positive integer $n$, we denote by $C_n$ the cyclic group of order $n$, and regard $C_n$ as a subgroup of $U(1)$ by identifying it with the group of $n$-th roots of unity. 
    \item Let $\mathrm{Sp}$ denote the stable infinity category of spectra. For a compact Lie group $G$, we denote the stable infinity category of genuine $G$-spectra by $\mathrm{Sp}^G$. In particular, $S \in \mathrm{Sp}^G$ denotes the sphere spectrum. 

    \item We denote the suspension spectrum functor by $\Sigma^\infty \colon \cS_* \to \Sp$, where $\cS_*$ is the category of pointed spaces. Similarly, we denote the suspension spectrum functor with a disjoint base point by $\Sigma^\infty_+ \colon \cS \to \Sp$.
    
    \item We denote by $\eta \in \pi_1 S$ and $\nu \in \pi_3 S $ the {\it integral} (not $2$-local) generators of $\pi_1 S \simeq \Z/2$ and $\pi_3 S \simeq \Z/24$, which are represented by the Hopf fibration for complex and quaternionic numbers, respectively.
    When we work $3$-locally in Section \ref{sec_C3}, we denote generators by $\alpha \in \pi_3 S_{(3)} \simeq \Z/3$ and $\beta = \langle \alpha, \alpha, \alpha \rangle \in \pi_{10} S_{(3)} \simeq \Z/3$. The localization map $S \to S_{(3)}$ sends $\nu$ to $\alpha$. 
    We use the same notation for the images of the elements $\pi_* S$ under the Hurewicz map $S \to \TMF$.

    \item For a compact Lie group $G$, $\RO(G)$ denotes the real representation ring of $G$. For each element $\tau \in \RO(G)$, we denote its representation sphere by $S^{\tau} \in \Sp^G$.
     \item For a compact Lie group $G$ and a genuine $G$-equivariant spectrum $E$, we denote by $E^G$ the genuine (categorical) $G$-fixed point spectrum of $E$. 
     
    \item Given an element $\tau \in \RO(G)$, we write 
    \begin{align}
        E[\tau] := E \otimes S^\tau \in \mathrm{Sp}^G. 
    \end{align}
    Its genuine $G$-fixed point spectrum $E[\tau]^G \in \mathrm{Sp}$ is the spectrum that represents the corresponding $\RO(G)$-graded $E$-cohomology theory. 
    \item 
    For a real $G$-representation $V$, we denote by
    \begin{align}\label{eq_notation_chi}
    	\chi(V) \in \Map(S^0, S^V)^{G}
    \end{align}
    the unique nontrivial $G$-equivariant map that sends $0 \mapsto 0$ and $\infty \mapsto \infty$. We also denote by the same symbol the $G$-equivariant map
    \begin{align}\label{eq_notation_chi_map}
    	\chi(V) := \id_E \otimes \chi(V) \colon E \to E \otimes S^V = E[V]
    \end{align}
    for any $G$-equivariant spectrum $E$. 
    Its homotopy class is called the {\it Euler class} associated with the representation $V$, and we again denote it by the same symbol $\chi(V) \in \pi_0 E[V]^G$.
    \item For an element $\tau \in \RO(G)$, let us write $\overline{\tau} := \tau - \dim(\tau) \cdot 1 \in \RO(G)$ where $1 = \underline{\R} \in \RO(G)$ is the class of the one-dimensional trivial representation. 
    \item We employ the following notations for the representations of interest in this paper: 
    \begin{itemize}
    	\item $\mu \in \RO( U(1))$: the fundamental (tautological) representation of $U(1)$, i.e., the real $2$-dimensional vector space $\R^2 \simeq \C$ with the complex multiplication.
    	\item $\rho_n \in \RO(C_n)$ for each positive integer $n$: the restriction of the fundamental representation of $U(1)$, 
    	\begin{align}
    		\rho_n:=\Res_{U(1)}^{C_n} \mu. 
    	\end{align}
    	\item $\lambda \in \mathrm{Rep}_O(C_2)$: the fundamental real $1$-dimensional representation. We have $\rho_2 \simeq 2 \lambda $.
        \item $V_G \in \Rep_O(G)$ for $G = \Spin(k), \Sp(k), \SU(k)$: the fundamental (a.k.a. vector or defining) representation of $G$. 
    \end{itemize}

	\item \label{notation_dual} Let $R$ be an $E_\infty$ ring spectrum. For a dualizable object $x \in \Mod_R$, we denote by $D_R(x) = \Hom_R(x, R)$ its dual in $\Mod_R$. $R$ is mostly $\TMF$ in this article, so we adopt the shorthand $D:= D_{\TMF}$. 

	\item  We use the following conventions for modular forms. We denote by $$\MF:= \Z[c_4, c_6, \Delta^\pm]/(c_4^3 - c_6^2 - 1728\Delta)$$ the ring of weakly holomorphic integral modular forms (i.e., holomorphic away from the cusps, and with integral Fourier coefficients in the variable $q =\exp(2\pi i \tau)$). 
    Capitalized ``Modular Forms'' means weakly holomorphic modular forms in this paper. 
    Denote $ \MF|_{\deg = m}$ be the set of weakly holomorphic modular forms of weight $\frac{m}{2}$. In particular, we have the edge homomorphism
    \begin{align}
    	e_{\MF} \colon \pi_{m} \TMF  \to  \MF|_{\deg = m}. 
    \end{align}
    \item \label{notation_JF} We use the conventions for Jacobi forms following \cite{dabholkar2014quantum,Gristenko2020382}. 
    We denote by $\mathfrak{H} := \{\tau \in \C \mid \mathrm{Im}(\tau ) > 0 \}$ the upper half space of the complex plane. 
    For each $k \in \Z_{\ge 0}$ and $w \in \Z$, consider holomorphic functions on $(z, \tau) \in \C \times \mathfrak{H}$ satisfying the transformation properties
    \begin{align}
    	\phi\left(\frac{a\tau + b}{c \tau + d}, \frac{z}{c\tau + d} \right) &= (c\tau + d)^w e^{\frac{\pi i k cz^2}{c\tau + d}} \phi(\tau, z),  \\
    	\phi(\tau, z + \lambda \tau + \mu) &= e^{-\pi i  k(\lambda ^2 \tau + 2 \lambda z) } \phi(\tau, z),
    \end{align}
    for all $\begin{pmatrix}
    	a & b \\ c & d 
    \end{pmatrix} \in SL(2, \Z) $ and $(\lambda, \mu ) \in \Z^2$, and Fourier expansions
    \begin{align}
    \phi(q, y) = \sum_{r \in \Z + \frac{k}{2}} \sum_{n \ge N} c(n, r) q^n y^r,
    \end{align}
    for some integer $N$, where $(q, y) = (\exp(2\pi i \tau), \exp(2\pi i z))$. 
    \begin{itemize}
    	\item Such functions are called {\it weakly holomorphic} Jacobi forms of index $\frac{k}{2}$ and weight $w$.
    	\item If $c(n, r) \neq 0$ only when $n \ge 0$, then such functions are called {\it weak} Jacobi forms. In addition, if $\phi$ satisfies $c(n, r) \neq 0$ only when $r^2 \ge 4kn$, then $\phi$ is called a {\it holomorphic} Jacobi form. We do not handle weak or holomorphic Jacobi forms in this paper.
    	\item If all Fourier coefficients $c(n,r)$ are integers, we add the adjective {\it integral} in all the above cases.
    \end{itemize}
    In the text, we capitalize the first letters in ``Jacobi Forms'' to mean {\it weakly holomorphic Jacobi forms} and denote by $\JF_k$ the set of all integral Jacobi Forms of index $\frac{k}{2}$. We put the $\Z$-grading on $\JF_k$ so that $\JF_k|_{\deg = m}$ consists of Jacobi Forms with weight $w = -k+\frac{m}{2}$. This convention makes $\JF_k$ a $\Z$-graded module over the $\Z$-graded ring $\MF$. As will be recalled in Section \ref{sec_preliminaries}, we have a canonical map
    \begin{align}
    	e_{\JF} \colon \pi_{m}\TJF_k = \pi_m \Gamma(\mathcal{E}^\ori; \mathcal{O}_{\mathcal{E}^\ori}(ke)) \to    \JF_k|_{\deg = m}. 
    \end{align}
    \item \label{notation_a} For notational ease, we employ the following notations for three of the generators of the $\Z$-graded ring $\oplus_k \JF_k$ of integral Jacobi Forms:
    \begin{align}
    	a &:= \phi_{-1, \frac12} = 	\frac{\theta_{11}(z, q)}{\eta^3(q)}  
    	=(e^{\pi i z} - e^{-\pi i z}) \prod_{m \ge 1} \frac{(1-q^m e^{2\pi i z})(1-q^me^{-2\pi i z})}{(1-q^m)^2} \in \JF_1|_{\deg = 0},  \label{eq_notation_a} \\
    	b & := \phi_{0, 1} = -\frac{3}{\pi^2}\wp(z, q)\frac{\theta^2_{11}(z, q)}{\eta^6(q)}   \in \JF_2|_{\deg = 4}, \label{eq_notation_b} \\
    	c &:= \phi_{0, \frac{3}{2}} = \frac{\theta_{11}(2z, q)}{\theta_{11}(z, q)} \in \JF_3|_{\deg = 6}, \label{eq_notation_c}
    \end{align}
    where the notation $\phi_{w, k}$ follows \cite{Gritsenko:1999fk}. 
  
\end{itemize}

\section{Preliminaries}\label{sec_preliminaries}

\subsection{Gepner-Meier's genuine equivariant $\TMF$}\label{subsec_GepnerMeier}

We briefly review the genuine equivariant $\TMF$ developed by Lurie and Gepner-Meier, referring to \cite{GepnerMeier} for complete details.  
Spectral algebraic geometry, as introduced and explored by Lurie in \cite{LurieElliptic1, LurieElliptic2, LurieElliptic3},  provides a conceptual framework of elliptic cohomology with integral coefficients. 
Denote by $\cE^\ori \to \cM^\ori$ the universal oriented elliptic curve over $\cM^\ori$ as a spectral Deligne-Mumford stack (the term ``spectral algebraic'' is henceforth often omitted). Then, the spectrum \textit{Topological Modular Forms} is defined to be the global section of the structure sheaf
\begin{align}\label{eq_SAG_ellcoh}
	\TMF := \Gamma(\cM^\ori; \mathcal{O}_{\cM^{\ori}} ) \in \CAlg.
\end{align}

Gepner and Meier refined $\TMF$ to a genuine $G$-equivariant spectrum for compact Lie groups $G$. They constructed the {\it equivariant elliptic cohomology functor }
\begin{align}\label{eq_Ell_GM}
	\Ell \colon \mathcal{S}_{\rm Orb} \to \mathrm{Shv}(\cM^\ori), 
\end{align}
where $\cS_{\rm Orb}$ is the category of orbispaces and $\mathrm{Shv}(\cM^\ori)$ is the sheaf category on the big \'{e}tale site on $\cM^\ori$. Note that a stack over $\cM^\ori$ can be regarded as a sheaf by taking the corepresented functor.
The category $\cS_{\rm Orb}$ has topological stacks $\bfB G = [*//G]$ for each compact Lie group $G$ as objects. The image $\Ell(\bfB A) $ for each compact abelian Lie group $A$ is defined to be the hom stack $\Hom(\hat{A}, \cE)$, where $\hat{A}$ is the Pontryagin dual of $A$.  In particular, 
\begin{align}\label{eq_EllBU(1)=E}
	\Ell(\bfB U(1)) \simeq \mathcal{E}^\ori , \quad  \Ell(\bfB C_n) \simeq \cE^\ori[n]
\end{align}
where $\cE^\ori[n] \subset \cE^\ori$ is the $n$-torsion of elliptic curves. The functor \eqref{eq_Ell_GM} is given by the left Kan extension from abelian group cases. In general, $\Ell(\bfB G)$ can be regarded as a spectral algebraic counterpart of the complex analytic moduli stack of flat $G$-bundles on the dual elliptic curve.

Moreover, for each compact Lie group $G$, by restricting the domain $\cS^G_* \subset \cS_{\Orb}$ and extending the target of $\Ell$ to the category of quasicoherent sheaves on $\Ell(\bfB G)$, they obtained a colimit-preserving functor
\begin{align}\label{eq_EllG}
	\widetilde{\cEll}_G \colon \mathcal{S}^{G}_* \to \QCoh(\Ell(\bfB G))^\op.  
\end{align}
Composing with the global section functor $\Gamma$, we obtain a colimit-preserving functor
\begin{align}\label{eq_GammaEllG}
	\Gamma	\widetilde{\cEll}_G \colon \mathcal{S}^G_* \to \Spectra^\op, \quad X \mapsto \Gamma(\Ell(\bfB G); \widetilde{\cEll}_G(X)).
\end{align}
They showed that the functor \eqref{eq_GammaEllG} is represented by a genuine $G$-spectrum, also denoted by $\TMF \in \Spectra^G$, and demonstrated its functoriality with respect to $G$. In this construction, we can identify the global section of sheaves with the $G$-equivariant cohomology
\begin{align}\label{eq_def_TMFcoh}
	\lMap_G(X, \TMF)^G \simeq  \Gamma(\Ell(\bfB G); \widetilde{\cEll}_G(X)).
\end{align}
In particular, we have
\begin{align}
	(\TMF)^G \simeq \Gamma(\Ell(\bfB G); \mathcal{O}_{\Ell(\bfB G)}).
\end{align}

For each virtual representation $V \in \RO(G)$, we denote its $V$-shift as
\begin{align}
    \TMF[V] \coloneqq \TMF \otimes S^{V} \in \Sp^G
\end{align}
and its $\RO(G)$-graded $\TMF$ homology as
\begin{align}\label{eq_def_ROGTMF}
	\TMF[V]^G \coloneqq  (\TMF \otimes S^V)^G = \TMF(S^{-V})^G = \Gamma(\Ell(\bfB G), \widetilde{\cEll}(S^{-V})). 
\end{align}

An essential feature of genuine equivariant $\TMF$ is dualizability: \footnote{By contrast, equivariant $\KU$ is not dualizable.}
\begin{fact}[{Dualizability of $\TMF^G$ \cite{GepnerMeierNEW}}]\label{fact_dualizability_TMF}
	For any compact Lie group $G$, $\TMF^G$ is dualizable in $\Mod_\TMF$, and its $\TMF$-dual $D(\TMF^G)$ is equivalent to $\TMF[-\Ad (G)]^G$.  
\end{fact}

Consequently, for not only inclusion but {\it any} Lie group homomorphism $f \colon G \to H$, we can define the {\it transfer map along $f$}
\begin{align}\label{eq_tr_f}
	\tr_f \colon	\TMF[-\Ad(G)]^G \to \TMF[-\Ad(H)]^H 
\end{align}
to be the dual of the restriction map
$\res_f \colon \TMF^H \to \TMF^G$.
Given $(G, H)$, if a unique or natural map $f: G \to H$ exists such that its choice is unambiguous within the context, we often write $\res_f$ as $\res_H^G$ and $\tr_f$ as $\tr_G^H$.

We also note that, for every $V \in \RO(G)$, $\TMF[V]^G \in \Mod_\TMF$ has the dual
\begin{align}\label{eq_twisted_dual}
	D(\TMF[V]^G) \simeq \TMF[-V-\Ad(G)]^G, 
\end{align}
with the evaluation map 
\begin{align}\label{eq_coev_duality_TMF}
	\TMF[V]^G \otimes \TMF[-V -\Ad(G)]^G \xrightarrow{\rm multiplcation} \TMF[-\Ad(G)]^G \xrightarrow{\tr_G^e} \TMF. 
\end{align}

This paper mainly considers the cases $G= U(1), C_n, Sp(1)$. The $U(1)$- and $Sp(1)$-equivariant $\TMF$ are more accessible and their structures are well understood; see the following subsection and \cite[Appendix~A and B]{lin2024topologicalellipticgenerai}. 
Our main objective is to analyze the $C_n$-equivariant $\TMF$ by leveraging the $U(1)$-equivariant case. 

\subsection{$U(1)$-equivariant $\TMF$ $=$ Topological Jacobi Forms}\label{subsec_TJF}

Here we summarize the theory of $U(1)$-equivariant $\TMF$. The $\RO(U(1))$-graded $\TMF$ are also known as {\it Topological Jacobi Forms}, as they are the spectral refinements of the graded ring of weakly holomorphic Jacobi forms. 

\begin{defn}[{$\TJF_k$}]\label{def_TJF}
	For each integer $k$, we define
	\begin{align}\label{eq_def_TJF}
		\TJF_k := \TMF[k\mu]^{U(1)} \simeq \Gamma(\cE^\ori; \cO_{\cE^\ori}(ke)) , 
	\end{align}
	where $\mu \in \RO_{O}(U(1))$ is the fundamental representation of $U(1)$. 
\end{defn}
The second equivalence in \eqref{eq_def_TJF} follows from the identification.
\begin{align}
    \widetilde{\cEll}_{U(1)}(S^\mu) \simeq \cO_{\cE^\ori}(-e)
\end{align}
obtained via the cofiber sequence in $\cS^{U(1)}_*$
\begin{align}\label{eq_chi_sigma}
	U(1)_+ \to S^0 \to S^{{\mu}}.
\end{align} 
As $\widetilde{\mathcal{E}ll}_{U(1)}$ is symmetric monoidal \cite{GepnerMeier}\footnote{In general, the functor $\widetilde{\cEll}_G$ is \textit{not} symmetric monoidal, even in the case of $G=C_n$. Symmetric monoidality is one of the reasons why $U(1)$-equivariant $\TMF$ behaves well. }, it sends the Spanier-Whitehead dual $S^{-\mu}$ to $ \cO_{\cE^\ori} (e)$ and therefore $S^{-k\mu}$ to $\cO_{\cE^\ori} (ke)$.
Note that the tensor product on $\cO_{\cE^\ori} (ke)$ induces the multiplication
\begin{align}
	``\cdot" ~ \colon \TJF_k \otimes_\TMF \TJF_s \to \TJF_{k+s}, 
\end{align}
and the $\Z$-graded spectra $\{ \TJF_k \}_{k \in \Z}$ become an $\mathbb{E}_2$-ring object. 

$\{ \TJF_k \}_{k \in \Z_{\ge0}}$ can be regarded as the spectral refinement of Jacobi forms by the following observation. Using the flatness of the map $\cE^\ori \to \cM^\ori$, we can show that the homotopy sheaf $\pi_{2m}\cO_{\cE^\ori} (ke)$ is isomorphic to $p^* \omega^m \otimes \cO_{\cE}(ke)$ as sheaves on the underlying stacks, which is the universal elliptic curve $\cE$. We consider a further base change to the complex universal elliptic curve $\cE_\C$. Recall that the function $a= \phi_{-1, \frac12} \in \JF_1|_{\deg=0}$ in \eqref{eq_notation_a} precisely vanishes at $z=0$ with order $1$. Therefore, multiplication by $a$ yields an isomorphism of line bundles
\begin{align}
    \cO_{\cE_\C}(ke) \otimes \omega^m \simeq L_{m, 2k}
\end{align}
where $L_{m,2k}$ is the Looijenga line bundle in $\cE_\C$ defined in \cite{looijenga_root_1976}. See \cite{BauerMeierTJF} for further explanation. As the global section of the Looijenga line bundle is the set of Jacobi forms, we obtain an isomorphism
\begin{align}\label{eq_ak}
	a^k \cdot ~ \colon \Gamma(\cE_{\C}; \cO_{\cE_\C}(ke) \otimes \omega^{m}) \simeq \JF_k^\C|_{\deg = 2m} . 
\end{align}
The edge homomorphism of the descent spectral sequence for $\TJF_k$ gives a map to the group of integral Jacobi forms of index $\frac{m}{2}$
\begin{align}
	e_\JF \colon \pi_\bullet \TJF_k \to  \JF_k|_{\deg = \bullet},
\end{align}
and the following diagram commutes:
\begin{align}\label{diag_TJF_JF}
	\xymatrix{
		\pi_\bullet \TJF_k \ar[rr]^-{e_\JF}  \ar@{=}[d] && \JF_k |_{\deg = \bullet} \ar@{_{(}->}[d]^-{a^{-k} \cdot}\\
		\pi_\bullet \Gamma(\cE, \cO_\cE(ke)) \ar[rr]^-{(\cE_\C \to \cE)^*} &&\Gamma(\cE_{\C}; \cO_{\cE_\C}(ke) \otimes \omega^{\bullet/2})
	}
\end{align}
Therefore, the graded $\mathbb{E}_2$-ring spectra $\{ \TJF_k \}_{k \in \Z_{\ge0}}$ are the spectral refinements of the graded ring of integral Jacobi forms. 

The equivariant Euler class of the fundamental representation
\begin{align}
	\chi(\mu) \in \pi_0 \TJF_1
\end{align}
has the image
\begin{align}\label{eq_char_a}
	e_\JF(\chi(\mu)) = a \in \JF_1|_{\deg=0}. 
\end{align}
It is an essential ingredient of the {\it stabilization-restriction fiber sequence}: for each $k$,
\begin{align}\label{eq_TJF_stabres}
	\TMF[2k-1] \xrightarrow{\tr_e^{U(1)}}\TJF_{k-1} \xrightarrow{\chi(\mu) \cdot} \TJF_{k} \xrightarrow{\res_{U(1)}^e} \TMF[2k].
\end{align}
The middle arrow $\chi(\mu) \cdot$ corresponds to the canonical map $\cO_{\cE^\ori}((k-1)e) \to \cO_{\cE^\ori}(ke) $. 
By \eqref{diag_TJF_JF}, the following diagram commutes:
\begin{align}\label{diag_stabres_JF_TJF}
	\xymatrix{
		\pi_\bullet	\TJF_{k-1} \ar[r]^-{\chi(\mu)\cdot} \ar[d]^-{e_\JF}&\pi_\bullet \TJF_{k} \ar[rrr]^-{\res_{U(1)}^e} \ar[d]^-{e_\JF}&&&\pi_{\bullet-2k} \TMF \ar[d]^-{e_\MF} \\
		\JF_{k-1}|_{\deg = \bullet} \ar[r]^-{a \cdot} & \JF_{k}|_{\deg = \bullet} \ar[rrr]^-{\res_{z=0} \colon \phi(z, q) \mapsto \phi(0, q)} &&&\MF|_{\deg = \bullet - 2k} 
	}
\end{align}

In fact, the stabilization-restriction fiber sequence \eqref{eq_TJF_stabres} for $U(1)$-equivariant $\TMF$ is a special case of the more general construction in \cite[Proposition~4.45]{lin2024topologicalellipticgenerai}. 
In this paper, we heavily use another special case, that of the $C_2$-equivariant $\TMF$:
\begin{align}\label{eq_C2_stabres}
		\TMF[k-1] \xrightarrow{\tr_e^{C_2}}\TMF[(k-1)\lambda]^{C_2} \xrightarrow{\chi(\lambda) \cdot} \TMF[k\lambda]^{C_2} \xrightarrow{\res_{C_2}^e} \TMF[k], 
\end{align}
where $\lambda \in \RO(C_2)$ is the fundamental real $1$-dimensional representation of $C_2$. 
The cases of $Sp(1)$ and $\Spin(k)$ also appear in \eqref{eq_TEJF_stabres} and \eqref{eq_stabres_spin} below. 

The $\TMF$-module structure of $\TJF_k$ for $k \ge 0$ is well understood. 
\begin{fact}[Bauer-Meier \cite{BauerMeierTJF}]\label{fact_TJF_cellstr}
Let $\tr \colon \Sigma \bC P^\infty_+ \to S^0$ be the circle-equivariant transfer map and $q \colon \bC P^{\infty}_+ \to S^0$ be the trivial map. Define a stable cell complex $P_k$ for $k \in \Z_{\ge0}$ by
\begin{align}\label{eq_def_Pm}
	P_k := \cofib\left( \Sigma \CP^{k-1}_+ \xrightarrow{\tr \oplus \Sigma q} S^0 \oplus S^1 \right),
\end{align}
where the domain is restricted to $\Sigma \bC P^{k-1}_+$. Then we have an isomorphism of $\TMF$-modules
    \begin{align}\label{eq_TJF_cellstr}
        \TJF_k \simeq \TMF \otimes P_k.
    \end{align}
\end{fact}

\begin{rem}
    Recall that the cofiber of the transfer $\tr \colon \Sigma \bC P^{k-1} \to S^0$ is the stunted projective space $\Sigma^2 \bC P^{k-1}_{-1}$. The projection $S^0 \oplus S^1 \to S^0$ induces a map $P^{k} \to \Sigma^2 \bC P^{k-1}_+$ whose fiber is $S^1$. Therefore, $P_k$ can be thought of as ``$\Sigma^2 \bC P^{k-1}_{-1}$ with the 2-cell removed''.
\end{rem}

We can find the stable attaching maps of $\CP^k$ and, therefore, of $P_k$ in \cite{Mosher}.
In this paper, we mainly use 
\begin{align}
	P_0 & \simeq S^0 \oplus S^1  \\
	P_1 &\simeq S^0, \label{eq_P1} \\
	P_2 &\simeq S/\nu = S^0 \cup_\nu S^4, \label{eq_P2} \\
	P_3 &\simeq S^0 \cup_\nu S^4 \cup_\eta S^6, \label{eq_P3}\\
	P_4 & \simeq S^0 \cup_\nu S^4 \cup_{\eta \oplus 2\nu} (S^6 \oplus S^8).
    \label{eq_P4}
\end{align}
In the case of $k=0$, the isomorphism is explicitly given by \cite{GepnerMeier}
\begin{align}\label{eq_TJF_0}
	\res_e^{U(1)} \oplus \tr_{e}^{U(1)}	\colon \TMF \oplus \TMF[1] \to \TJF_0,
\end{align}
and in the case of $k=1$, it is given by
\begin{align}\label{eq_TJF_1}
\chi(\mu) \colon \TMF \simeq \TJF_1. 
\end{align}

The cell diagrams of $\TJF_k$ for $0 \le k \le 4$ are depicted in Figure \ref{celldiag_TJF}. 
Each dot labeled by an integer $n$ denotes a $\TMF$-cell of degree $n$. 

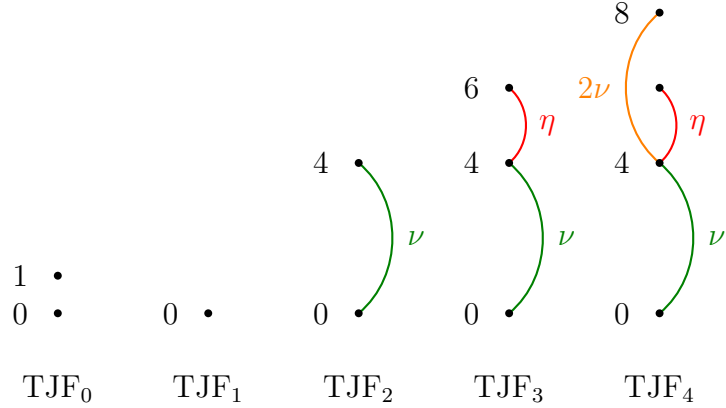
\begin{figure}[h]
	\centering
	\begin{tikzpicture}[scale=0.5]
		\begin{scope}[shift={(-4, 0)}]
			\begin{celldiagram}
				\n{0} \n{1}
				\foreach \y in {0,1} {
					\node [left] at (-0.5, \y) {$\y$};
				}
			\end{celldiagram}
			\node [] at (0, -2) {$\TJF_{0}$};
		\end{scope}
		
		\begin{scope}[shift={(0, 0)}]
			\begin{celldiagram}
				\n{0} 
				\foreach \y in {0} {
					\node [left] at (-0.5, \y) {$\y$};
				}
			\end{celldiagram}
			\node [] at (0, -2) {$\TJF_{1}$};
		\end{scope}
		
		\begin{scope}[shift={(4, 0)}]
			\begin{celldiagram}
				\nu{0}
				\n{0} \n{4} 
				\foreach \y in {0, 4} {
					\node [left] at (-0.5, \y) {$\y$};
				}
			\end{celldiagram}
			\node [] at (0, -2) {$\TJF_{2}$};
		\end{scope}
		\node [right, green!50!black] at (5, 2) {$\nu$};
		
		\begin{scope}[shift={(8, 0)}]
			\begin{celldiagram}
				\nu{0} \eta{4}
				\n{0} \n{4} \n{6} 
				\foreach \y in {0, 4, 6} {
					\node [left] at (-0.5, \y) {$\y$};
				}
			\end{celldiagram}
			\node [] at (0, -2) {$\TJF_{3}$};
		\end{scope}
		\node [right, green!50!black] at (9, 2) {$\nu$};
		\node [right, red] at (8.5, 5) {$\eta$};
		
		\begin{scope}[shift={(12, 0)}]
			\begin{celldiagram}
				\nu{0} \eta{4} \eightnu{4}
				\n{0} \n{4} \n{6} \n{8}
				\foreach \y in {0, 4, 8} {
					\node [left] at (-0.5, \y) {$\y$};
				}
			\end{celldiagram}
			\node [] at (0, -2) {$\TJF_{4}$};
		\end{scope}
		\node [right, green!50!black] at (13, 2) {$\nu$};
		\node [right, red] at (12.5, 5) {$\eta$};
		\node [left, orange] at (11, 6) {$2\nu$};
	\end{tikzpicture}
	\caption{The cell diagrams of $\TJF_k$.}\label{celldiag_TJF}
\end{figure}

The isomorphism in Fact~\ref{fact_TJF_cellstr} is compatible with the stabilization-restriction fiber sequence \eqref{eq_TJF_stabres} because the following diagram commutes:
\begin{align}\label{eq_attaching_P}
	\xymatrix{
		\TJF_{k-1} \ar[r]^-{\chi(\mu)} & \TJF_k \ar[r]^-{\res_{U(1)}^e} & \TMF[2k] \\
		P_{k-1} \ar[u]^-{\TMF \otimes - } \ar@{^{(}->}[r] & P_k \ar@{->>}[r] \ar[u]^-{\TMF \otimes - } & S^{2k} \ar[u]^-{\TMF \otimes - } 
	}
\end{align}
where the bottom row is the cofiber sequence induced by the inclusion $\CP^{k-2} \hookrightarrow \CP^{k-1}$ (see \eqref{eq_def_Pm}). 

$\TJF_{k}$ for negative $k$ follows from applying duality to Fact~\ref{fact_TJF_cellstr}. In this paper, $D$ and $D_S$ denote the duals in $\Mod_\TMF$ and $\Spectra$, respectively (see Section~\ref{subsec_notations}). We obtain
\begin{align}\label{eq_duality_TJF}
	\TJF_k \simeq D(\TJF_{-k})[1] \simeq \TMF \otimes D_S(P_{-k})[1]
\end{align}
by the dualizability of the equivariant $\TMF$ in \eqref{eq_twisted_dual}. 
For example, setting $k=-1$, we get 
	\begin{align}\label{eq_TJF-1}
	\TJF_{-1} \stackrel{\eqref{eq_duality_TJF}}{\simeq} D(\TJF_1)[1] \stackrel{\eqref{eq_TJF_1}}{\simeq} \TMF[1].  
\end{align}

Based on the computation \cite{Tominaga}, we set the notations for the elements in $\TJF_k$ to be used later.

\begin{defn}\label{fact_pi_TJF}~
	\begin{enumerate}
        \item \label{fact_pi7TJF2} We fix a generator of $\pi_7 \TJF_2 \simeq \bZ / 12$ and denote it by $\zeta$. Note that the map 
        \begin{align}
		\res_{U(1)}^e \colon \pi_7 \TJF_2 \to \pi_3  \TMF \simeq \Z\nu /24\nu
		\end{align}
        sends $\zeta$ to $2\nu$.
		\item \label{fact_pi7TJF3}
		 The group $\pi_7 \TJF_3 \simeq \Z /6$ is generated by
		\begin{align}
			\gamma := \chi(\mu) \cdot \zeta \in \pi_7 \TJF_3. 
		\end{align}
		In other words, the stabilization map $\chi(\mu) \cdot \colon \TJF_2 \to \TJF_3 $ induces a surjection
		\begin{align}
			\chi(\mu) \cdot \colon \pi_7 \TJF_2 = \Z\zeta/12\zeta \twoheadrightarrow \pi_7 \TJF_3 = \Z\gamma/6\gamma, \quad \zeta \mapsto \gamma. 
		\end{align}
	\end{enumerate}
\end{defn}
\begin{rem}\label{fact_TJF3toTJF4}
	The image of the unit $1 \in \pi_0 \TMF$ under the transfer map $\tr_e^{U(1)} \colon \TMF[7] \to \TJF_3$ is equal to $\gamma$.
	Therefore, the first map in the stabilization-restriction fiber sequence \eqref{eq_TJF_stabres} for $k=4$ can be regarded as the multiplication with $\gamma$
	\begin{align}\label{eq_stabres_TJF4}
		\TMF[7] \xrightarrow{\gamma \cdot }\TJF_3 \xrightarrow{\chi(\mu) \cdot} \TJF_4
	\end{align}
	in the graded $\mathbb{E}_2$-ring structure of $\TJF$.
\end{rem}

\subsection{$Sp(1)$-equivariant $\TMF$ $=$ Topological Even Jacobi Forms}\label{subsec_TEJF}

The $Sp(1)$-equivariant $\TMF$ is also understood. Our $\RO(C_2)$-graded $\TMF$ results require the $Sp(1)$-equivariant $\TMF$ (see \ref{sec_twisted_C2}), so we provide a brief overview. See \cite[Appendix~B]{lin2024topologicalellipticgenerai} for more details. 

Denote by $V_{Sp(1)} \in \RO(Sp(1))$ the real $4$-dimensional fundamental representation of $Sp(1)$. We adopt the name {\it Topological Even Jacobi Forms} for the $Sp(1)$-equivariant $\TMF$, and follow the grading convention
\begin{align}
	\TEJF_{2k} := \TMF[kV_{Sp(1)}]^{Sp(1)}, \quad k \in \Z.
\end{align}
We do not define $\TEJF_{k'}$ for odd $k'$. This grading convention ensures that the restriction along the inclusion $U(1) \hookrightarrow Sp(1)$ gives a map
\begin{align}\label{eq_res_Sp_U}
\res_{Sp(1)}^{U(1)} \colon \TEJF_{2k} \to \TJF_{2k}. 
\end{align}

The stabilization-restriction fiber sequence for $\TEJF$ becomes
\begin{align}\label{eq_TEJF_stabres}
\TMF[4k-1] \xrightarrow{\tr_e^{Sp(1)}}	\TEJF_{2k-2} \xrightarrow{\chi(V_{Sp(1)})} \TEJF_{2k} \xrightarrow{\res_{Sp(1)}^{e}} \TMF[4k].
\end{align}
Analogous to Fact~\ref{fact_TJF_cellstr}, there exists an even more straightforward identification of the $\TMF$ cell structures for $\TEJF$:
\begin{fact}\label{fact_cellstr_TEJF}
	We have an isomorphism of $\TMF$-modules
	\begin{align}
		\TEJF_{2k} \simeq \TMF \otimes \Sigma^{-4}\Sigma^{\infty}\HP^{k+1}. 
	\end{align}
	 Furthermore, we have a commutative diagram similar to \eqref{eq_attaching_P}, using the inclusion $\HP^{k} \hookrightarrow \HP^{k+1}$.
\end{fact}
In particular, we obtain the descriptions
\begin{align}
	\TEJF_0 &\simeq \TMF, \\
	\TEJF_2  &\simeq \TMF /\nu, \\
	\TEJF_4 &\simeq \TMF \otimes \left(S^0 \cup_\nu S^4 \cup_{2\nu} S^8 \right). \label{eq_TEJF4}
\end{align}

The restriction map \eqref{eq_res_Sp_U} is an isomorphism for $k=1$, 
\begin{align}
\res_{Sp(1)}^{U(1)} \colon \TEJF_2 \simeq \TJF_2 \stackrel{\eqref{eq_P2}}{\simeq} \TMF/\nu. 
\end{align}
Also, by the cell structure \eqref{eq_TEJF4}, we get
\begin{fact}\label{fact_zeta_TEJF}
	 The image of the unit through the map $\tr_e^{Sp(1)} \colon \TMF[7] \to \TEJF_2$ is equal to the element $\zeta \in \pi_{7} \TEJF_2 = \pi_7 \TJF_2 \simeq \Z\zeta/12\zeta$ defined in Fact \ref{fact_pi_TJF} \eqref{fact_pi7TJF2}.
     By the stabilization-restriction fiber sequence \eqref{eq_TEJF_stabres}
\begin{align}\label{eq_stabres_TEJF4}
	\TMF[7] \xrightarrow{\zeta \cdot }\TEJF_2 \xrightarrow{\chi(V_{Sp(1)}) \cdot} \TEJF_4,
\end{align}
we regard $\TEJF_4$ as the cofiber of the multiplication by $\zeta$.
\end{fact}

\subsection{$C_n$-equivariant $\TMF$}\label{subsec_prelim_CnTMF}
$\RO(C_n)$-graded $\TMF$ is the focus of this paper. For each $V \in \RO(C_n)$, the $C_n$-equivariant elliptic cohomology functor \eqref{eq_EllG} produces a sheaf
\begin{align}
\cL(-V) \in \QCoh(\cE^{\ori}[n]), 
\end{align}
where we used the identification $\Ell(\bfB C_n) = \cE^\ori [n]$ in \eqref{eq_EllBU(1)=E}, and thus we have
\begin{align}\label{eq_prelim_TMFCn}
	\TMF[V]^{C_n} = \Gamma(\cE^\ori [n], \cL(-V)). 
\end{align}
We also consider the base change to elliptic curves over $\C$. For each $V \in \RO(C_n)$, the complex analytic counterpart of \eqref{eq_prelim_TMFCn} is denoted by
\begin{align}\label{eq_MFCn}
	\MF[V]^{C_n}_\C := \Gamma(\cE_\C[n], \omega^\bullet \otimes \cL_\C(-V)),
\end{align}
and we call elements of this $\MF_\C$-module {\it $V$-twisted $C_n$-equivariant Modular Forms}. 
In the literature, elements in \eqref{eq_MFCn} are described in two equivalent ways:
\begin{itemize}
	\item One is as {\it $\Gamma_1(n)$-Modular Forms with multiplier systems}, or {\it Modular Forms with level-$n$ structures and multiplier systems}. These are holomorphic functions $\phi(q)$ in the upper half-plane with covariance under the transformation by the congruence subgroup $\Gamma_1(n) \subset SL_2(\Z)$. 
	The ``multiplier system'' refers to the constants in the $\Gamma_1(n)$ covariance formula, which indicates the twist by the line bundle $\cL_\C(-V)$.
	\item The other view is as vector-valued Modular Forms, using the identification \begin{align}
		\Gamma(\cE_\C[n]; \cL_\C(-V)) \simeq \Gamma(\moduli_\C; p_! \cL_\C(-V)), 
	\end{align} where $p \colon \cE_\C[n] \to \moduli$ is the $n^2$-fold covering map, $p_! \cL_\C(-V)$ is the fiberwise direct sum, and $p_! \cL_\C(-V)$ is a sheaf of rank $n^2$ over $\moduli$.
\end{itemize}
This paper avoids an explicit formula for multiplier systems or the vector-valued modular forms transition function (see \cite{GrigolettoPutrov}). From the models discussed, the integrality of $C_n$-equivariant vector-valued modular forms is understood by ensuring that the Fourier coefficients of each vector component (in $q = e^{2\pi i \tau}$) are integral.\footnote{In the literature, integral $\Gamma_1(n)$-Modular Forms typically only require the Fourier coefficients at the $i\infty$ cusp to be integral, but not at the other cusps, which are the $SL(2,\Z)$ images of $i\infty$. In other words, integrality is imposed on only one component of the corresponding vector-valued modular form.
}
We denote by $\MF[V]^{C_n}$ the submodule of $\MF[V]^{C_n}_\C$ consisting of {\it integral} $C_n$-equivariant $V$-twisted Modular Forms. 
The $C_n$-equivariant $\TMF$ induces a map 
\begin{align}
	e_{\MF(n)} \colon \pi_\bullet \TMF[V]^{C_n} \to \MF[V]^{C_n}
\end{align}
for each $V \in \RO(C_n)$.\footnote{This fact follows by factoring through the $C_n$-equivariant Tate $K$-theory \cite{Ganter} \cite{Luecke}. }

We use the standard inclusion $C_n \hookrightarrow U(1)$. Through the equivariant elliptic cohomology functor \eqref{eq_Ell_GM}, this corresponds to the inclusion $\iota_n \colon \cE^\ori[n] \hookrightarrow \cE^{\ori}$ of $n$-torsion points. For every $U(1)$-representation $W \in \RO(U(1))$, we have $\iota_n^* \cL(-W) \simeq \cL(-\res_{U(1)}^{C_n} W)$, and the following diagram commutes:
\begin{align}\label{diag_res_U(1)_Cn}
	\xymatrix{
	\TMF[W]^{U(1)} \ar[r]^-{\res_{U(1)}^{C_n}} \ar@{=}[d] & \TMF[\res_{U(1)}^{C_n} W]^{C_n} \ar@{=}[d] \\
	 \Gamma(\cE^\ori; \cL(-W)) \ar[r]^-{\iota_n^*}&  \Gamma(\cE^\ori [n]; \iota_n^* \cL(-W))	  
	}
\end{align}

\subsection{Equivariant sigma orientations}\label{subsec_sigma}

In \cite{ando2010multiplicative}, an $E_\infty$-ring map
\begin{align}
    \MString \to \TMF
\end{align}
was constructed and coined the {\it sigma orientation} of $\TMF$. In particular, they established the Thom isomorphism in $\TMF$-cohomology for vector bundles with string structures. From both a mathematical perspective \cite{lurie2009survey} and a physical perspective \cite[Appendix A]{tachikawa2023anderson}, there is a prevalent expectation that the sigma orientation extends to the genuine equivariant $\TMF$. Such an equivariant orientation would imply the Thom isomorphism statement for $\RO(G)$-graded $\TMF$: Given an element $V \in \RO(G)$ that possesses a string structure, that is, a null homotopy $\mathfrak{s}$ of the composition
\begin{align}
    BG \xrightarrow{\overline{V}} BO \to P^4 BO,
\end{align}
we expect an isomorphism of $G$-equivariant $\TMF$-module spectra
\begin{align}
    \sigma(\mathfrak{s}) \colon \TMF[V]^G \simeq \TMF.
\end{align}
Although the genuine equivariant sigma orientation is not yet fully established, a partial result is available that addresses the requisite groups. For further information, consult \cite[Section 2.3]{lin2024topologicalellipticgenerai}. The $G=U(1)$ case is essential in this paper. To state the result, consider the following:

\begin{prop}\label{prop_U1twist}
	We have a non-split short exact sequence
	\begin{align}\label{eq_U1twist}
		\xymatrix{
			0 \ar[r] &  H^4(BU(1); \Z) \ar@{=}[d] \ar[r] &  [BU(1), P^4 BO] \ar[d]^-{\simeq} \ar[r] &  H^2(BU(1); \Z/2) \ar@{=}[d] \ar[r] & 0 \\
			&\Z \{ c_1^2 \} \ar[r]^-{c_1^2 \mapsto 2} & \Z \ar[r]^-{1 \mapsto c_1} & \Z /2 \{ c_1 \}&
		}. 
	\end{align}
	The generator in the middle term is given by the fundamental representation $\overline{\mu} \colon BU(1) \to BO$ composed with the truncation $BO \to P^4BO$. 
\end{prop}
\begin{proof}
	Note that the fiber sequence \eqref{eq_U1twist} is induced by the fibration in the Whitehead tower
	\begin{align}
		K(\Z, 4) \to P^4 BO \to P^2 BO.
	\end{align}
	We may replace the domain $BU(1)$ with $\bC P^2$ because the inclusion $\bC P^2 \hookrightarrow BU(1)$ is $5$-connected. The Atiyah-Hirzebruch spectral sequence tells us that $[\bC P^2, BO] = \widetilde{KO^0}(\bC P^2)$ is isomorphic to $\Z$, generated by the tautological line bundle. However, the map between short exact sequences
	\begin{align}
		\xymatrix{
			[S^2, P^4BO] \ar[d]^{\simeq} & [\bC P^2, P^4 BO] \ar[l] \ar[d] & [S^4, P^4BO] \ar[l] \ar[d]^{\simeq} \\
			[S^2, BO]  & [\bC P^2, BO] \ar[l] & [S^4, BO] \ar[l]
		}
	\end{align}
	shows that $[\bC P^2, P^4 BO] \simeq [\bC P^2, BO] \simeq \Z$, so we get the desired conclusion.
\end{proof}

We use Meier's result \cite{MeierDual}. 

\begin{fact}[{Equivariant Thom isomorphism for $U(1)$-equivariant $\TMF$}]\label{fact_sigma_U(1)}
	Given two elements $V, W \in \RO(U(1))$ with \begin{align}\label{eq_fact_sigma_U(1)}
		[\overline{V}] = [\overline{W}] \quad \mbox{in} \quad [BU(1), P^4 BO], 
	\end{align} there is a unique $U(1)$-equivariant $\TMF$-module spectra isomorphism
    \begin{align}
		\TMF[V] \simeq \TMF[W]. 
	\end{align}
\end{fact}

\begin{rem}
	The uniqueness in Fact~\ref{fact_sigma_U(1)} follows because \eqref{eq_fact_sigma_U(1)} shows that $V-W$ has a $BU\langle 6 \rangle$-structure, and the choice of $BU \langle 6 \rangle$-structure is unique up to homotopy as seen from $[BU(1), \Omega BU \langle 6 \rangle] = 0$. 
\end{rem}
Let
\begin{align}
	\varphi_n := (-)^n \colon U(1) \to U(1)
\end{align}
be the $n$-th power group homomorphism. When $n$ is positive, the kernel of $\varphi_n$ is the $n$-th cyclic group $C_n \subset U(1)$. 
\begin{lem}\label{lem_twistn2}
	The homomorphism $\varphi_n$ induces a homomorphism (see Proposition \ref{prop_U1twist})
	\begin{align}
		\res_{\varphi_n} = n^2 \cdot \colon [BU(1), P^4 BO] \simeq \Z \to [BU(1), P^4 BO] \simeq \Z. 
	\end{align}
\end{lem}
\begin{proof}
	This follows from Proposition \ref{prop_U1twist} and the fact that $\varphi_n^* c_1 = n \cdot c_1$ in $H^2(BU(1); \Z) \simeq \Z[c_1]$. Thus 
    \begin{align}
		\varphi_n^* = n^2 \cdot \colon  H^4(BU(1); \Z) \to H^4(BU(1); \Z). 
	\end{align}
\end{proof}
This lemma implies that we have a canonical isomorphism for any pair of integers $n$ and $k$
\begin{align}
	  \TMF[\mu^n + k\mu]^{U(1)} \simeq \TMF[(n^2+k) \mu -2n^2+2]^{U(1)} \simeq \TJF_{n^2+k}[-2n^2 + 2], \label{eq_sigma_prelim}
\end{align}
where $\mu^n:= \res_{\varphi_n} \mu$ for shorthand.

Fact~\ref{fact_sigma_U(1)} implies the Thom isomorphism for $G=C_p$ at each prime $p$. Let $\rho_n :=  \res_{U(1)}^{C_n} \mu \in \RO(C_n)$ represent the class $[\overline{\rho_n}] \in [BC_n, P^4 BO]$. The abelian group structure of $[BC_n, P^4 BO]$ and the class $[\overline{\rho_n}]$ are examined in \cite{GrigolettoPutrov}. Notably, if $n=p$ is prime, we have:
\begin{itemize}
	\item For $p=2$, the fundamental real $1$-dimensional representation $\lambda \in \mathrm{Rep}_O(C_2)$ generates the group $[BC_2, P^4 BO] \simeq \Z/8$, and the element $[\overline{\rho_2}]$ is twice the generator,  
	\begin{align}\label{eq_C2_twist}
		[\overline{\rho_2}] = 2 [\overline{\lambda }] \in [BC_2, P^4 BO] \simeq \Z/8 \{[\overline{\lambda}] \}. 
	\end{align}
	\item For any odd prime $p>2$, the element $[\overline{\rho_p}]$ generates the group $[BC_p, P^4 BO] \simeq \Z/p$,
	\begin{align}\label{eq_Cp_twist}
		[BC_p, P^4 BO] = \Z / p \{ [\overline{\rho_p}] \}. 
	\end{align}
\end{itemize}

The following proposition confirms that the periodicity of the $\RO(C_p)$-grading of $\TMF$ aligns with \eqref{eq_C2_twist} and \eqref{eq_Cp_twist}:
\begin{prop}[{Periodicity of the $\RO(C_p)$-graded $\TMF$}]\label{prop_sigma_Cp}~
	\begin{enumerate}
		\item We have an isomorphism of $C_2$-equivariant $\TMF$-module spectra, 
		\begin{align}
			\TMF[8\overline{\lambda}] = \TMF[4\overline{\rho_2}] \stackrel{\sigma}{\simeq} \TMF. 
		\end{align}
		\item For each odd prime $p$, we have an isomorphism of $C_p$-equivariant $\TMF$-module spectra, 
		\begin{align}
		 \TMF[p\overline{\rho_p}] \stackrel{\sigma}{\simeq} \TMF. 
		\end{align}
	\end{enumerate}
\end{prop}
\begin{proof}
	Consider the isomorphism of $U(1)$-equivariant $\TMF$-module spectra from Lemma \ref{lem_twistn2} and Fact \ref{fact_sigma_U(1)}:
	\begin{align}\label{eq_proof_Cp_periodicity}
		\TMF[\overline{\mu^n}] \stackrel{\sigma}{\simeq} \TMF[n^2\overline{\mu}]. 
	\end{align}
 (1) is derived by applying the functor $\res_{U(1)}^{C_2} \colon \Spectra^{U(1)} \to \Spectra^{C_2}$ to \eqref{eq_proof_Cp_periodicity} for $n=2$. For (2), we have \begin{align}\label{eq_proof_Cp_periodicity2}
		\res_{U(1)}^{C_p} \mu^{n} \simeq \res_{U(1)}^{C_p} \mu^{p-n} \quad \mbox{in } \RO(C_p) 
	\end{align} as real representations for any $n \in \Z$. This leads to an equivalence of $C_p$-equivariant $\TMF$-module spectra, \begin{align}
		\TMF[p\overline{\rho_p}] = \TMF\left[ \left( \frac{p+1}{2} \right)^2\overline{\rho_p} -  \left( \frac{p-1}{2} \right)^2\overline{\rho_p} \right] 
	\stackrel{\sigma}{\simeq} \TMF\left[ 	\res_{U(1)}^{C_p} \left( \overline{\mu^{\frac{p+1}{2} } }- \overline{\mu^{\frac{p-1}{2}}} \right)  \right] \simeq \TMF, 
	\end{align} where the middle isomorphism is the specialization \eqref{eq_proof_Cp_periodicity} to $n=p$, and the right isomorphism is \eqref{eq_proof_Cp_periodicity2}.
\end{proof}

\section{General strategy : setting up fiber sequences}\label{sec_general}

\subsection{Untwisted cases}\label{subsec_general_nontwisted}

Consider the following diagram of pointed $U(1)$-spaces:
\begin{align}\label{seq_sigma^n}
	\xymatrix{
	& S^\mu \ar[rd]^-{\varphi_n} & \\
	\Ind_{C_n}^{U(1)}(S^0) \simeq (U(1) / (C_n))_+ \ar[r] & S^0 \ar[u]^-{\chi(\mu)} \ar[r]_-{\chi(\mu^n)} & S^{\mu^n}, 
	}
\end{align} 
The notation $\chi(-)$ is in \eqref{eq_notation_chi}. Here, $\mu$ denotes the fundamental representation in $\RO(U(1))$, and $\mu^n := \res_{\varphi_n} \mu$. The triangle is commutative, and the horizontal sequence forms a cofiber sequence of pointed $U(1)$-spaces.
$U(1)$-equivariant $\TMF$-homology $(\TMF \otimes -)^{U(1)}$ yields the following fiber sequence in $\Mod_\TMF$:
\begin{align}\label{eq_CnU1}
    (\TMF \otimes \Ind_{C_n}^{U(1)}(S^0))^{U(1)} \xrightarrow{\tr_{C_n}^{U(1)}} \TMF^{U(1)} \xrightarrow{\chi(\mu^n)} \TMF[\mu^n]^{U(1)}.
\end{align}
By the isomorphisms
\begin{align}
    (\TMF \otimes \Ind_{C_n}^{U(1)}(S^0))^{U(1)} &\simeq  \TMF[1]^{C_n}, \\
    \TMF^{U(1)} &=\TJF_0,
\end{align}
and \eqref{eq_sigma_prelim}, the fiber sequence \eqref{eq_CnU1} can be reformulated as
\begin{align}\label{eq_fundamental}
	\TMF[1]^{C_n} \xrightarrow{\tr_{C_n}^{U(1)}}  \TJF_0 \xrightarrow{\chi(\mu^n)} \TJF_{n^2}[-2n^2 + 2] \xrightarrow{\res_{U(1)}^{C_n}} \TMF[2]^{C_n}. 
\end{align}

\begin{prop}\label{prop_general_nontwist}
We have the following commutative diagram of $\TMF$-modules:

\begin{align}
    \xymatrix{
    \TMF^{C_n}/\tr_{e}^{C_n}[1] \ar[rr]\ar@/^10pt/[dd] && \TJF_1  \ar[r]^-{\res_{\varphi_n}} & \TJF_{n^2}[-2n^2+2] \ar@{=}[dd] \\ 
    &\TMF \ar[ru]^-{\simeq} \ar[rd]^-{\res_{e}^{U(1)}} &&\\
    \TMF^{C_n}[1] \ar@/^10pt/[d]^-{\tr_{C_n}^e} \ar[uu] \ar[rr]^-{\tr_{C_n}^{U(1)}} && \TJF_0 \ar@/^10pt/[d]^-{\tr_{U(1)}^e} \ar[r]^-{\chi(\mu^n)}  \ar[uu]_-{\chi(\mu) \cdot} &\TJF_{n^2}[-2n^2+2]
     \\
\TMF[1] \ar@{=}[rr] \ar[u]^-{\tr_e^{C_n}} &&  \TMF[1], \ar[u]^-{\tr_{e}^{U(1)}} & 
    }
\end{align}
where the horizontal sequences are fiber sequences, the vertical sequences are split fiber sequences, 
and the middle row is the fiber sequence in \eqref{eq_fundamental}. 
If we define
\begin{align}
    \widetilde{\TMF^{C_n}} := \cofib\left(\TMF \xrightarrow{\tr_e^{C_n}} \TMF^{C_n}\right), 
\end{align}
then the fiber sequence 
\begin{align}\label{seq_tildeTMFCn}
    \TMF \xrightarrow{\tr_e^{C_n}} \TMF^{C_n} \to \widetilde{\TMF^{C_n}}
\end{align}
splits, providing a canonical isomorphism 
\begin{align}\label{eq_TMFCn_fib}
    \widetilde{\TMF^{C_n}}[1] \simeq  \mathrm{fib}\left( \TMF \xrightarrow{\res_e^{U(1)}} \TJF_0  \xrightarrow{ \chi(\mu) } \TJF_1 \xrightarrow{\res_{\varphi_n} } \TJF_{n^2}[-2n^2+2] \right).
\end{align}

\end{prop}
\begin{proof}
  	The middle vertical split fiber sequence arises from \eqref{eq_TJF_0} and \eqref{eq_TJF_1}.
	The rest of the diagram follows automatically. 
\end{proof}

\begin{prop}\label{prop_a(nz)}
    The image of the unit $1 \in \pi_0 \TMF$ through the composition \eqref{eq_TMFCn_fib} 
    \begin{align}
        \TMF \xrightarrow{\res_e^{U(1)}} \TJF_0  \xrightarrow{ \chi(\mu) } \TJF_1 \xrightarrow{\res_{\varphi_n} } \TJF_{n^2}[-2n^2+2]
    \end{align}
    is equal to the element $\res_{\varphi_n} (\chi(\mu)) = \chi(\mu^n) \in \TJF_{n^2}|_{\deg = 2n^2-2}$. This element satisfies
    \begin{align}
        e_\JF \left(\chi(\mu^n)\right)
        = a(nz) = \frac{\theta_{11}(nz, \tau)}{\eta(\tau)^3}  \in \JF_{n^2}|_{\deg = 2n^2-2}. 
    \end{align}
\end{prop}
\begin{proof}
    This follows from $e_\JF(\chi(\mu)) = a $ as seen in \eqref{eq_char_a} and the fact that the group homomorphism $\varphi_n \colon U(1) \to U(1)$ induces the $n$-fold map of the universal oriented elliptic curve.
\end{proof}

In summary:

\begin{cor}\label{cor_untwisted}
We have an isomorphism of $\TMF$-modules, 
\begin{align}
\TMF^{C_n} \simeq \TMF \oplus \widetilde{\TMF^{C_n}}
\end{align}
with 
\begin{align}
	\widetilde{\TMF^{C_n}} := \cofib\left(\TMF \xrightarrow{\tr_e^{C_n}} \TMF^{C_n}\right) \simeq \mathrm{cofib} \left( \TMF[-2] \xrightarrow{\chi(\mu^n)} \TJF_{n^2}[-2n^2]\right). 
\end{align}

\end{cor}

The fiber sequence \eqref{eq_fundamental} is related to operations in Jacobi forms and $C_n$-equivariant modular forms, as the next proposition shows.
\begin{prop}\label{prop_JF_fibsec_untwisted}
	The following diagram commutes:
 \begin{align}
 	\xymatrix@C=5em{
 		\MF^{C_n}|_{\deg = m-1} \ar[r]^-{\sum_{\cE_\C[n]/\moduli_\C}} & \MF|_{\deg = m-1} && \\
 		\pi_m \TMF[1]^{C_n} \ar[u]^-{e_{\MF(n)}} \ar[r]^-{\tr_{C_n}^{U(1)}} & \pi_m  \TMF^{U(1)} \ar[u]_-{e_\MF \circ \tr_{U(1)}^e} \ar[r]^-{\chi(\mu^n)} \ar[d]_-{e_\MF \circ \res_{U(1)}^e}& \pi_m \TJF_{n^2}[-2n^2 + 2] \ar[d]^-{e_\JF} \ar[r]^-{\res_{U(1)}^{C_n}} & \pi_m \TMF[2]^{C_n}  \ar[d]^-{e_{\MF(n)}}\\
 	& \MF|_{\deg = m} \ar[r]^-{a(nz)\cdot} & \JF_{n^2}|_{\deg = m+2n^2 -2}  \ar[r]^-{\iota_n^*} & \MF^{C_n}|_{\deg = m-2}
 		}
 \end{align}
 The middle row is the long exact sequence from the fiber sequence \eqref{eq_fundamental}. The top horizontal arrow represents the fiberwise sum along the $n^2$-fold covering map $p \colon \cE_\C[n] \to \moduli_\C$, and the bottom right arrow indicates the restriction along the inclusion $\iota_n \colon \cE_\C[n] \hookrightarrow \cE_\C$.
\end{prop}

\begin{proof}
The top square is commutative because the transfer map along $C_n  \to e$ is, through the elliptic cohomology functor, given by the counit map $p_! p^* \cO_\moduli \to \cO_\cM$ of the adjunction $p_! \dashv p^*$ \cite[Section~7.4]{LurieElliptic3}. 
The bottom left square commutes by Proposition \ref{prop_a(nz)}, and the bottom right by \eqref{diag_res_U(1)_Cn}.
\end{proof}

\subsection{Twisted cases}\label{subsec_twisted}
We note the following fact:
\begin{lem}\label{lem_IndRes}
    Let $H \subset G$ be an inclusion of compact Lie groups. 
    For any $G$-spectra $X$, we have 
    \begin{align}
        \Ind_{H}^G \circ \Res_G^H (X) \simeq (G/H)_+ \wedge X. 
    \end{align}
\end{lem}

Applying Lemma \ref{lem_IndRes}, we obtain the isomorphism of $U(1)$-spectra for each $k \in \Z$:
\begin{align}
    \Ind_{C_n}^{U(1)}(S^{k\rho_n}) \simeq (U(1)/C_n)_+ \wedge S^{k\mu}.  
\end{align}
Then wedging $S^{k\mu}$ to the cofiber sequence \eqref{seq_sigma^n}, we get the following cofiber sequence of $U(1)$-spectra:
\begin{align}\label{seq_cofib_twisted}
    \Ind_{C_n}^{U(1)}(S^{k\rho_n})  \to S^{k\mu} \xrightarrow{\chi(\mu^n)} S^{\mu^n+k\mu}. 
\end{align}
Again applying $U(1)$-equivariant $\TMF$-homology $(\TMF \otimes -)^{U(1)}$ to the fiber sequence \eqref{seq_cofib_twisted}, we get a fiber sequence
\begin{align}
    (\TMF \otimes \Ind_{C_n}^{U(1)}(S^{k\rho_n}))^{U(1)} \to \TJF_k \xrightarrow{\chi(\mu^n)} \TMF[\mu^n + k\mu]^{U(1)}. 
\end{align}
We have
\begin{align}
	(\TMF \otimes \Ind_{C_n}^{U(1)}(S^{k\rho_n}))^{U(1)} &\simeq  \TMF[k\rho_n + 1]^{C_n}, \\
	\TMF[\mu^n + k\mu]^{U(1)} &\simeq \TJF_{k+n^2}[-2n^2+2], 
\end{align} 
where the second equivalence is by \eqref{eq_sigma_prelim}. 
We get

\begin{prop}\label{prop_twisted}
	Let $n$ be a positive integer and $k$ be any integer. 
	We have the following fiber sequence in $\Mod_\TMF$:
	\begin{align}\label{seq_twistedZn_TJF}
		\TMF[k\rho_n + 1]^{C_n} \xrightarrow{\tr_{C_n}^{U(1)}} \TJF_k \xrightarrow{\chi(\mu^n)} \TJF_{k+n^2}[-2n^2+2] \xrightarrow{\res_{U(1)}^{C_n}} \TMF[k\rho_n + 2]^{C_n}. 
	\end{align}
	Moreover, the following diagram commutes:
	\begin{align}\label{eq_stabilization}
		\xymatrix{
		\TMF[k\rho_n + 1]^{C_n} \ar[r] \ar[d]^-{\chi(\rho_n) \cdot} & \TJF_k \ar[r]^-{\chi(\mu^n)} \ar[d]^-{\chi(\mu) \cdot}& \TJF_{k+n^2}[-2n^2+2] \ar[d]^-{\chi(\mu) \cdot} \\
		\TMF[(k+1)\rho_n + 1]^{C_n} \ar[r] & \TJF_{k+1} \ar[r]^-{\chi(\mu^n)} & \TJF_{k+1+n^2}[-2n^2+2]  
		}
	\end{align}
\end{prop}

We can describe the relation to the corresponding operations on Jacobi Forms and $C_n$-equivariant Modular Forms. 

\begin{prop}
	The following diagram commutes:
	 \begin{align}
		\xymatrix{
		 \pi_m  \TJF_k  \ar[r]^-{\chi(\mu^n)} \ar[d]_-{e_\JF}& \pi_m \TJF_{k+n^2}[-2n^2 + 2] \ar[d]^-{e_\JF} \ar[r]^-{\res_{U(1)}^{C_n}} & \pi_m \TMF[k\rho_n + 2]^{C_n}  \ar[d]^-{e}\\
			 \JF_k|_{\deg = m} \ar[r]^-{a(nz)\cdot} & \JF_{k+n^2}|_{\deg = m+2n^2 -2}  \ar[r]^-{\res_{\cE[n]}} & \MF[k\rho_n]^{C_n}|_{\deg = m-2}
		}
	\end{align}
\end{prop}

\begin{proof}
	The left square commutes by Proposition \ref{prop_a(nz)}, and the right square commutes by \eqref{diag_res_U(1)_Cn}.
\end{proof}

Let us fix a positive integer $n$. 
So far, we have constructed the fiber sequence \eqref{seq_twistedZn_TJF} for each $k \in \Z$. 
This family of fiber sequences is self-dual in $\Mod_\TMF$ by recalling that we have the following duality relations in $\Mod_\TMF$ via \eqref{eq_twisted_dual}:
\begin{align}
	\TJF_k &\simeq D(\TJF_{-k}[-1]), \label{eq_duality_TJF_2} \\
	\TMF[k\rho_n]^{C_n} &\simeq D(\TMF[-k\rho_n]^{C_n}). \label{eq_duality_TMFCn}
\end{align}

\begin{prop}\label{prop_duality_seq}
	Let $n$ be any positive integer and $k$ be any integer. 
	The following diagram commutes: 
	\begin{align}
		\xymatrix{
			\TMF[k\rho_n + 1]^{C_n} \ar[r]^-{\tr_{C_n}^{U(1)}} \ar[d]^-{\simeq}_-{\eqref{eq_duality_TMFCn} } & \TJF_k \ar[r]^-{\chi(\mu^n)} \ar[d]^-{\simeq}_-{\eqref{eq_duality_TJF_2} }& \TJF_{k+n^2}[-2n^2+2] \ar[r]^-{\res_{U(1)}^{C_n}}  \ar[d]^-{\simeq}_-{\eqref{eq_duality_TJF_2}}& \TMF[k\rho_n + 2]^{C_n}\ar[d]^-{\simeq}_-{\eqref{eq_duality_TMFCn} }  \\
			D(	\TMF[-k\rho_n - 1]^{C_n}) &D(\TJF_{-k}[-1]) \ar[l]^-{D(\res_{U(1)}^{C_n})}&D(\TJF_{-k-n^2}[2n^2-3]) \ar[l]^-{\chi(\mu^n)} &D( \TMF[-k\rho_n - 2]^{C_n})\ar[l]^-{D(\tr_{C_n}^{U(1)})} 
		}
	\end{align}
	Here the rows are the fiber sequences in \eqref{seq_twistedZn_TJF} for $k$ and the $\TMF$-linear dual to that for $-(k+n^2)$, where we have also used $n^2\rho_n = 2n^2\underline{\R}$ in $\RO(C_n)$. 
	In other words, the fiber sequences in \eqref{seq_twistedZn_TJF} for $k$ and $-(k+n^2)$ are dual to each other in $\Mod_\TMF$. 
\end{prop}

\begin{proof}
	The commutativity of the right and left squares follows from the fact that the dual of the restriction map is the transfer map. The middle square commutes because the graded multiplicative structure of $\{ \TJF_k \}_{k \in \mathbb{Z}}$ is natural with respect to the duality in Fact~\ref{fact_dualizability_TMF}.
\end{proof}

\subsection{Odd twisted case for $n=2$}\label{subsec_oddtwist_C2}

In the case $n=2$, the representation $\rho_2 \in \RO(C_2)$ is reducible, namely we have $2\lambda = \rho_2$ with the one-dimensional sign representation $\lambda \in \RO(C_2)$. As we have seen in \eqref{eq_C2_twist}, this representation $\lambda$ realizes the generator in the $8$-periodic classification of twists of $C_2$-equivariant $\TMF$.
Here we produce a fiber sequence similar to Proposition \ref{prop_twisted} that applies to $\TMF[n\lambda]^{C_2}$ with odd $n$. 

\begin{lem}
	Define $A \colon S^{\mu^2} \to \Sigma \Ind_{C_2}^{U(1)}(S^\lambda)$ to be a map of pointed $U(1)$-spaces given by the composition
	\begin{align}
	A \colon	S^{\mu^2} \xrightarrow{\rm{Cof}(\chi(\mu^2))} \Sigma(U(1)/C_2)_+ =	\Sigma \Ind_{C_2}^{U(1)}(S^0)  \xrightarrow{\Ind_{C_2}^{U(1)}(\chi(\lambda)) }\Sigma \Ind_{C_2}^{U(1)}(S^\lambda) , 
	\end{align}
	where the first arrow is the cofiber of the map $\chi(\mu^2) \colon S^0 \to S^{\mu^2}$. 
	Then $A$ is identified with the cofiber of $\varphi_2 \colon S^\mu \to S^{\mu^2}$ so that the following is a commutative diagram of fiber sequences in $\Sp^{U(1)}$:
	\begin{align}
		\xymatrix@C=7em{
		S^0 \ar[d]^-{\chi(\mu)} \ar[r]^-{\chi(\mu^2)} & S^{\mu^2} \ar@{=}[d] \ar[r]^-{\rm{Cof}(\chi(\mu^2))} &	\Sigma(U(1)/C_2)_+  =	\Sigma \Sigma^\infty \Ind_{C_2}^{U(1)}(S^0) \ar[d]_-{\Ind_{C_2}^{U(1)}(\chi(\lambda)) } \\
		S^\mu \ar[r]^-{\varphi_2} & S^{\mu^2} \ar[r]^-{A} & \Sigma \Ind_{C_2}^{U(1)}(S^\lambda) 
		}
	\end{align}
	where the top horizontal sequence is the fiber sequence \eqref{seq_sigma^n}. 
\end{lem}

\begin{proof}
Consider the cofiber sequence of pointed $C_2$-spaces
\begin{align}
	(C_2)_+ \to S^0 \xrightarrow{\chi(\lambda)} S^\lambda. 
\end{align}
Applying $\Ind_{C_2}^{U(1)}$ gives a cofiber sequence of pointed $U(1)$-spaces, 
\begin{align}
	\Ind_{C_2}^{U(1)}((C_2)_+ )  = (U(1))_+ \xrightarrow{\varphi_2} \Ind_{C_2}^{U(1)} (S^0) = (U(1)/C_2)_+ \xrightarrow{\Ind_{C_2}^{U(1)}(\chi(\lambda))} \Ind_{C_2}^{U(1)} (S^\lambda). 
\end{align}
Consider the following commutative diagram in $\Sp^{U(1)}$:
\begin{align}\label{diag_lem_oddtwist}
	\xymatrix{
	 S^0 \ar@{=}[r] \ar[d]^-{\chi(\mu)}& S^0  \ar[d]^-{\chi(\mu^2)} & \\
	S^\mu \ar[r]^-{\varphi_2} \ar[r] \ar[d] & S^{\mu^2} \ar[d]|-{\rm{Cof}(\chi(\mu^2))} \ar[rrd]^-{A}&\\
	\Sigma (U(1))_+  \ar[r]^-{\varphi_2}&  \Sigma (U(1)/C_2)_+  \ar[rr]_-{\Ind_{C_2}^{U(1)}(\chi(\lambda))} && \Sigma \Ind_{C_2}^{U(1)} (S^\lambda) 
	}, 
\end{align}
where the two vertical sequences are cofiber sequences and the bottom horizontal sequence is the cofiber sequence above. 
By this diagram, we get that $A$ is the fiber of $\varphi_2 \colon S^\mu \to S^{\mu^2}$ as desired. 
\end{proof}

To state the result, we make the following observation:

\begin{prop}\label{prop_Chua}
	~
	\begin{enumerate}
		\item The element $\chi(\mu^2) \in \pi_6 \TJF_4$ decomposes as
		\begin{align}
			\chi(\mu^2) = \chi(\mu) \cdot \{c\}, 
		\end{align}
		where $\chi(\mu) \in \pi_0 \TJF_1$ as before, and $\{c\} \in \pi_6 \TJF_3$ is the element whose image under the map $e_{\JF} \colon \pi_6 \TJF_3 \to \JF_3|_{\deg = 6}$ is $c \in \JF_3|_{\deg = 6}$ in \eqref{eq_notation_c}. Note that the map $e_{\JF} \colon \pi_6 \TJF_3 \to \JF_3|_{\deg = 6}$ is an isomorphism \cite{Tominaga}.                                                           
		\item Let $[\varphi_2] \in \pi_0(S^{\mu^2 - \mu})^{U(1)}$ be the element specified by the map $\varphi_2 \colon S^{\mu} \to S^{\mu^2}$ of $U(1)$-spheres. 
		Then the unit map $u \colon S^{\mu^2 - \mu} \to (\TMF \otimes S^{\mu^2 - \mu})^{U(1)} \simeq \TJF_3[6]$ sends
	$[\varphi_2]$ to $\{c\}$. 
	\end{enumerate}
\end{prop}

\begin{proof}
	The first claim follows from the equation $\frac{a(2z)}{a(z)} = c$ and the fact that both maps in $\pi_{6}\TJF_3 \xrightarrow{\chi(\mu)} \pi_6 \TJF_4 \xrightarrow{e_\JF} \JF_4|_{\deg = 6}$ are isomorphisms (see \cite{Tominaga}). 
    The second claim follows from the commutativity of the upper square of \eqref{diag_lem_oddtwist}. 
\end{proof}

We repeat the previous subsection's procedure by wedging $S^{k\rho_2}$ to diagram \eqref{diag_lem_oddtwist} and applying $U(1)$-equivariant $\TMF$-homology, yielding the following result:

\begin{prop}\label{prop_oddtwisted}
	Let $k$ be an integer. 
	We have a fiber sequence in $\Mod_\TMF$,
	\begin{align}\label{seq_oddtwisted}
		\TMF[(2k+1)\lambda + 1]^{C_2} \xrightarrow{} \TJF_{k+1} \xrightarrow{\{c\}} \TJF_{k+4}[-6] \xrightarrow{\chi(\lambda) \circ \res_{U(1)}^{C_2}} \TMF[(2k+1)\lambda + 2]^{C_2}. 
	\end{align}
	Moreover, the following diagram of fiber sequences commutes:
	\begin{align}\label{eq_twisted_stabilization}
		\xymatrix{
			\TMF[2k\lambda + 1]^{C_2} \ar[r]^-{\tr_{C_2}^{U(1)}} \ar[d]^-{\chi(\lambda) \cdot} & \TJF_{k} \ar[r]^-{\chi(\mu^2)} \ar[d]^-{\chi(\mu) \cdot}& \TJF_{k+4}[-6] \ar@{=}[d] \ar[r]^-{\res_{U(1)}^{C_2}} &\TMF[2k\lambda +2]^{C_2} \ar[d]^-{\chi(\lambda) \cdot} \\
			\TMF[(2k+1)\lambda + 1]^{C_2} \ar[r] & \TJF_{k+1} \ar[r]^-{\{c\}} & \TJF_{k+4}[-6]  \ar[r] & \TMF[(2k+1)\lambda + 2]^{C_2}
		}
	\end{align}
	Here, the upper row is the fiber sequence in Proposition \ref{prop_twisted}. 
\end{prop}

\section{Application 1 : $C_2$-equivariant $\TMF$}\label{sec_app1}

This section applies the general strategy developed in Section \ref{sec_general} to the case $n=2$. 
As we have seen in \eqref{eq_C2_twist} and Proposition \ref{prop_sigma_Cp} (1), the group $[BC_2, P^4 BO] \simeq \Z/8$ is generated by the class $[\lambda]$ of the fundamental real $1$-dimensional representation $\lambda \in \RO(C_2)$, and the $\RO(C_2)$-graded $\TMF$ satisfies the corresponding periodicity 
\begin{align}
	\TMF[8\overline{\lambda}]^{C_2} = \TMF[4\overline{\rho_2}]^{C_2} \simeq \TMF^{C_2}.  
\end{align}
In this section, we analyze each of the $8$ cases and determine their $\TMF$-module structures. 
Section~\ref{sec_untwist_C2} deals with the untwisted and the $\pm \lambda$-twisted cases, and Section~\ref{sec_twisted_C2} considers the rest. 
The main results are presented in Theorems~\ref{thm_nontwist_C2} and~\ref{thm_C2_twisted}. 
An elegant pattern of cell diagrams is shown in Figure~\ref{celldiag_TMFZ/2}. 

\subsection{Untwisted and $\pm \lambda$-twisted cases}\label{sec_untwist_C2}

For $C_n = C_2$, we have an identification
\begin{align}
	\widetilde{\TMF^{C_2}} := \cofib\left( \TMF \xrightarrow{\tr_e^{C_2}} \TMF^{C_2}\right) \simeq \TMF[\lambda]^{C_2} 
\end{align}
by the stabilization-restriction fiber sequence \eqref{eq_C2_stabres} for $C_2$. This means that Corollary~\ref{cor_untwisted} is written as
\begin{align}\label{eq_TMFC2_dec}
	\TMF^{C_2} \simeq \TMF \oplus \TMF[\lambda]^{C_2} 
\end{align}
with
\begin{align}\label{eq_1/8_cofib}
	\TMF[\lambda]^{C_2} \simeq \mathrm{cofib} \left( \TMF[-2] \xrightarrow{\chi(\mu^2)} \TJF_{4}[-8]\right). 
\end{align}

We can see the equivalence \eqref{eq_1/8_cofib} also by applying Proposition~\ref{prop_oddtwisted} for $k=0$. 
So the problem is reduced to understanding the homotopy type of the cofiber of $\chi(\mu^2)$. We note that we are not localizing at any prime. In particular, we reproduce and integrally refine the result by Chua \cite{chua}, who deduces the cell structure of $\TMF^{C_2}$ after localizing at prime $2$ by a computational method. 

To state the main result, we need to look into the element $\{c\} \in \pi_6 \TJF_3$ a little more. 
Recall that we have $\TJF_3 \simeq \TMF \otimes P_3$ with the cell complex $P_3 \simeq S^0 \cup_\nu S^4 \cup_\eta S^6$ (Fact \ref{fact_TJF_cellstr}, \eqref{eq_P3}).
\begin{prop}\label{prop_c}
	\begin{enumerate}
		\item We have $\pi_6 P_3 \simeq \Z$, and the image of the map
		\begin{align}\label{eq_P3_S6}
			\pi_6 P_3 \to \pi_6 P_3/P_2 \simeq \pi_6 S^6 = \Z
		\end{align}
		is $2\Z$. 
		\item The unit map
		\begin{align}\label{eq_P3_TJF3}
			\pi_6 P_ 3 \xrightarrow{u} \pi_6 \TMF \otimes P^3 \simeq \pi_6 \TJF_3
		\end{align}
		is injective, and its image is $\{c\} \cdot \Z$. 
		Denote by $\hat{c} \in \pi_6 P_3$ the unique element that maps to $\{c\}$ by the map \eqref{eq_P3_TJF3}. 
		This element maps to $2$ by \eqref{eq_P3_S6}. 
	\end{enumerate}
\end{prop}

\begin{proof}
	The first claim can be verified from the long exact sequence of homotopy groups. 
	The second claim follows from the commutative diagram
	\begin{align}\label{diag_P3TJF3}
		\xymatrix@C=5em{
		\pi_6 P_3 \ar@{>->}[r] \ar[d]^-{u} & \pi_6 S^6 \ar@{^{(}->}[d]^-{u} \\
		\pi_6 \TJF_3 = \{c\} \cdot \pi_0 \TMF \ar@{>->}[r]^-{\res_{U(1)}^e}_-{\{c\} \mapsto 2} \ar[d]^-{e_\JF}_-{\simeq} & \pi_0 \TMF \ar[d]^-{e_\MF}_-{\simeq} \\
		\JF_3 |_{\deg = 6} \ar@{>->}[r]^-{\res_{z=0}}_{c \mapsto 2} & \MF|_{\deg = 0}
		}
	\end{align}
	where we have used computation in \cite{Tominaga} and the equation $c = 2 + O(z)[[q]]$ to deduce that the horizontal arrows are injective and the left lower vertical arrow is an isomorphism. 
\end{proof}

By Proposition \ref{prop_c}, we have (see Figure~\ref{fig:cell_TJF3/c})
\begin{align}\label{eq_TJF3/c}
	\cofib(c) \simeq \TMF \otimes (S^0 \cup_\nu S^4 \cup_\eta S^6 \cup_2 S^7). 
\end{align}

\begin{thm}[{Cell structures of $\TMF^{C_2}$ and $\TMF[\pm\lambda]^{C_2}$} ]\label{thm_nontwist_C2}
   Let us define $C$ to be a finite spectrum (see Figure~\ref{fig:cell_C})
    \begin{align}
    	C := \cofib \left( S^{-2} \xrightarrow[\text{\tiny Prop.}\ref{prop_c} \ (2)]{\hat{c}} P_3[-8] \xhookrightarrow{\iota} P_4[-8] \right).
    \end{align}
     We have an equivalence
    \begin{align}
    		\TMF^{C_2} &\simeq \TMF \otimes(S^0 \oplus C), \label{eq_thm_nontwist_C2}\\
    	 \TMF[\lambda]^{C_2} &\simeq \TMF \otimes C, \label{eq_1/8}\\
    	 	\TMF[-\lambda]^{C_2} &\simeq \TMF \otimes D(C), \label{eq_7/8}
    \end{align}
\end{thm}

\begin{proof}
	For \eqref{eq_1/8},use \eqref{eq_1/8_cofib} and decompose
	\begin{align}
	\chi(\mu^2) \colon	\TMF[-2] \xrightarrow{\{c\}} \TJF_3[-8] \xrightarrow{\chi(\mu) \cdot} \TJF_4[-8] 
	\end{align}
	 using Propositions \ref{prop_Chua}, and identify the above composition with
	 \begin{align}
	 	\TMF \otimes S^{-2} \xrightarrow{\id \otimes \widehat{c}} \TMF \otimes P_3[-8] \xrightarrow{\id \otimes \iota} \TMF \otimes P_4 [-8]
	 \end{align}
	 using \eqref{eq_attaching_P} and \ref{prop_c}. This proves \eqref{eq_1/8}. From this,
	 \eqref{eq_thm_nontwist_C2} follows by \eqref{eq_TMFC2_dec}, and \eqref{eq_7/8} follows by the duality
	 \begin{align}
	 \TMF[-\lambda]^{C_2} \simeq D(\TMF[\lambda]^{C_2})
	 \end{align}
	 by \eqref{eq_duality_TMFCn}. This completes the proof. 
	 
\end{proof}

We can recover Chua's result from this Theorem.

\begin{cor}[{=\cite[Theorem 1.1]{chua}}]\label{CorChua}
    After $2$-localization, we have
    \begin{align}
    	C_{(2)} \simeq S^0 \oplus \cofib(c)[-8]. 
    \end{align}
  Thus we have an equivalence $$\TMF^{C_2}_{(2)} \simeq \TMF_{(2)} \otimes (S^0 \oplus S^0 \oplus \cofib(c)[-8])$$. 
\end{cor}

\begin{rem}\label{rem_computation_DL}
	In \cite{chua}, the $2$-local cell complex $\cofib(c)_{(2)} [-8]$ is denoted by $DL$ , and the homotopy groups of $\TMF\otimes DL$ are computed. 
\end{rem}

\begin{proof}
	This follows from the fact that, $2$-locally, the map
	\begin{align}
		\pi_{-1} S^{-4} \to \pi_{-4} \left( S^{-4} \cup_\eta S^{-2} \cup_2 S^{-1}\right) 
	\end{align}
	induced by the inclusion of the bottom cell sends $2\nu \in \pi_{-1}S^{-4}$ to zero. 
\end{proof}

\begin{figure}[h]
	\centering
	\begin{tikzpicture}[scale=0.5]
		\begin{celldiagram}
			\nu{0} \eta{4} \two{6}
			\n{0} \n{4}  \n{6} \n{7}  
			\foreach \y in {0,4,6,7} {
				\node [left] at (-1, \y) {$\y$};
			}
		\end{celldiagram}
		\node [right, red] at (1, 5) {$\eta$}; 
		\node [right, green!50!black] at (1, 2) {$\nu$};
		\node [blue] at (1, 6.5) {$2$};
	\end{tikzpicture}
	\caption{The cell diagram of $\cofib(c)$.}\label{fig:cell_TJF3/c}
\end{figure}
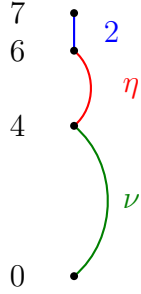

\begin{figure}[h]
  \centering
  \begin{tikzpicture}[scale=0.5]
     \begin{celldiagram}
     \nu{-8} \eta{-4} \eightnu{-4} \two{-2}
       \n{-8} \n{-4}  \n{-2} \n{-1}  \n{0} 
        \foreach \y in {-8, -4, -2, -1, 0} {
        \node [left] at (-1, \y) {$\y$};
      }
      \end{celldiagram}
       \node [left, orange] at (-1, -3) {$2\nu$};
       \node [right, red] at (1, -3) {$\eta$}; 
       \node [right, green!50!black] at (1, -6) {$\nu$};
       \node [blue] at (1, -1.5) {$2$};
      \end{tikzpicture}
     \caption{The cell structure of $C$.}\label{fig:cell_C}
\end{figure}
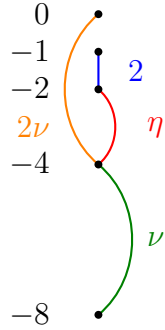

\subsection{$k\lambda$-twisted cases for $2 \le k \le 6$: the $C_2$ level-rank duality isomorphisms}\label{sec_twisted_C2}

In this subsection, we determine the structure of the $\RO(C_2)$-graded $\TMF$.
Combining the fiber sequence in Section~\ref{subsec_twisted} and ideas from the previous work \cite{lin2024topologicalellipticgenerai} by two of the authors, 
we prove the {\it level-rank duality} statement between $C_2 = O(1)$-equivariant $\TMF$ and $\Spin(k)$-equivariant $\TMF$ for small $k$, and as a result we compute the $RO(C_2)$-graded $\TMF$. 
Due to the periodicity
\begin{align}
	\TMF[(8+n)\overline\lambda]^{C_2} \simeq \Sigma^8 \TMF [n\overline\lambda]^{C_2}, 
\end{align}
it is sufficient to determine the structure of $\TMF[k\lambda]^{C_2}$ for $1\le k \le 7$. 

To state the result, we recall maps relating $C_2 = O(1)$-equivariant $\TMF$ with the $\Spin(k)$-equivariant $\TMF$ which was introduced in \cite[Definition~3.40 and (4.15)]{lin2024topologicalellipticgenerai}. For each $1 \le k \le 6$, we have a map of $\TMF$-module spectra
\begin{align}
	\cF_{(C_2)_k} \colon \TMF \to \TMF[kV_{C_2}]^{C_2} \otimes_\TMF \TMF[\overline V_{\Spin(k)}]^{\Spin(k)} = \TMF[k\lambda]^{C_2} \otimes_\TMF \TMF[\overline V_{\Spin(k)}]^{\Spin(k)}, 
\end{align} 
where $V_{C_2} \in \RO(C_2)$ and $V_{\Spin(k)} \in \RO(\Spin(k))$ are the orthogonal vector representations\footnote{For $\Spin(k)$, we are not using the spinor representation but the vector representation, i.e., the representation that factors through $\Spin(k) \to SO(k)$.}. 
It is defined as the composition
\begin{align}\label{eq_levelrank_coev}
	\cF_{(C_2)_k} \colon \TMF \xrightarrow{\chi(V_{\phi_k})} \TMF[V_{\phi_k}]^{C_2 \times \Spin(k)} \stackrel{\sigma(\Theta_k, \fraks)}{ \simeq } \TMF[kV_{C_2}]^{C_2} \otimes_\TMF \TMF[\overline V_{\Spin(k)}]^{\Spin(k)}, 
\end{align}
where we set
\begin{align}
	V_{\phi_k} := V_{C_2} \otimes_\R V_{\Spin(k) } \in \RO(C_2 \times \Spin(k)). 
\end{align}
In other words, $V_{\phi_k}$ has the underlying vector space $\R^k$ with commuting action by $C_2$ and $\Spin(k)$ given by the sign and the vector representations. 
The isomorphism $\sigma(\Phi_k, \fraks)$ is the equivariant Thom isomorphism in $\TMF$. The detailed explanation is in Remark \ref{rem_sigma} below. 
By the dualizability of equivariant $\TMF$, it is equivalently regarded as the map
\begin{align}\label{eq_levelrank}
	\cF'_{(C_2)_k} \colon  D(\TMF[\overline V_{\Spin(k)}]^{\Spin(k)}) \to \TMF[k\lambda]^{C_2} , 
\end{align}
which we call the {\it $(C_2)_k$ level-rank duality morphism}. 
Observe that the domain of \eqref{eq_levelrank} can be written by using the exceptional isomorphisms of spin groups and sigma orientations, as follows: 
\begin{prop}
	Using the isomorphisms of groups, 
		\begin{align}\label{eq_spin_grp}
		U(1) \stackrel{\varrho_2}{\simeq} \Spin(2), \   Sp(1),\stackrel{\varrho_3}{\simeq} \Spin(3),  \   Sp(1) \times Sp(1) \stackrel{\varrho_4}{\simeq}  \Spin(4), 
		Sp(2)  \stackrel{\varrho_5}{\simeq} \Spin(5), \   SU(4)\stackrel{\varrho_6}{\simeq}  \Spin(6)
	\end{align}
	we have the following Thom isomorphisms in equivariant $\TMF$:
	\begin{align}
		\TMF[\overline{V}_{\Spin(2)}]^{\Spin(2)}& \simeq \TMF[4\overline \mu]^{U(1)} = \TJF_4[-8], \label{eq_thom_spin2}\\
		\TMF[\overline{V}_{\Spin(3)}]^{\Spin(3)} &\simeq \TMF[2\overline{V}_{Sp(1)}]^{Sp(1)} = \TEJF_4[-8], \label{eq_thom_spin3}\\
			\TMF[\overline{V}_{\Spin(4)}]^{\Spin(4)} & \simeq \TMF[\overline V_{Sp(1)_L} \oplus \overline V_{Sp(1)_R}]^{Sp(1)_L \times Sp(1)_R} = \TEJF_2 \otimes_\TMF \TEJF_2[-4], \label{eq_thom_spin4}\\
				\TMF[\overline{V}_{\Spin(5)}]^{\Spin(5)} &\simeq \TMF[\overline{V}_{Sp(2)}]^{Sp(2)}, \label{eq_thom_spin5} \\
					\TMF[\overline{V}_{\Spin(6)}]^{\Spin(6)} &\simeq \TMF[\overline{V}_{SU(4)}]^{SU(4)}. \label{eq_thom_spin6}
	\end{align}
	In \eqref{eq_thom_spin4}, we wrote $Sp(1)_L \times Sp(1)_R = Sp(1) \times Sp(1)$ to distinguish the two copies of $Sp(1)$. 
\end{prop}

\begin{proof}
	Notice the equivalences of representations
	\begin{align}
		\res_{\varrho_2} (V_{\Spin(2)}) &\simeq \mu^2  \quad \mbox{in} \  \RO(U(1)), \\
			\res_{\varrho_3} ( V_{\Spin(3)} ) &\simeq \Ad(Sp(1)) \quad \mbox{in} \ \RO(Sp(1)), \\
				\res_{\varrho_4} (V_{\Spin(4)}) &\simeq V_{Sp(1)_L} \otimes V_{Sp(1)_R} \quad  \mbox{in} \ \RO(Sp(1)_L \times Sp(1)_R)  , \\
					\res_{\varrho_5^{-1}} (V_{Sp(2)}) &\simeq  \slashed S_{\Spin(5)}, \quad \mbox{in} \ \RO(\Spin(5)),\\
					\res_{\varrho_6^{-1}} (V_{SU(4)}) &\simeq \slashed S_{\Spin(6)}^+ \quad \mbox{in} \ \RO(\Spin(6)),
	\end{align}
	where we have denoted by $\slashed S_{\Spin(5)}$ and $\slashed S_{\Spin(6)}^+$ the (half)-spin representations. Then the isomorphisms \eqref{eq_thom_spin2}--\eqref{eq_thom_spin6} are given by the genuine equivariant sigma orientation of $\TMF$ \cite[Section~2.3]{lin2024topologicalellipticgenerai}.
	We claim the following equalities:
	\begin{align}
		[\mu^2] &= 4[\mu]   && \in [BU(1), P^4BO], \label{eq_proof_thom_spin2}\\
		[\Ad(Sp(1))] &= 2[V_{Sp(1)}]  && \in [BSp(1), P^4BO], \label{eq_proof_thom_spin3}\\
		[V_{Sp(1)_L} \otimes V_{Sp(1)_R}] &= [V_{Sp(1)_L} \oplus V_{Sp(1)_R}] && \in [BSp(1)_L \times BSp(1)_R, P^4BO], \label{eq_proof_thom_spin4}\\
		[ \slashed S_{\Spin(5)}] &= [V_{\Spin(5)}]  && \in [B\Spin(5), P^4BO],  \label{eq_proof_thom_spin5}\\
			[ \slashed S^+_{\Spin(6)}] &= [V_{\Spin(6)}]  && \in [B\Spin(6), P^4BO].  \label{eq_proof_thom_spin6}
	\end{align}
	Here, \eqref{eq_proof_thom_spin2} follows Lemma~\ref{lem_twistn2}, and the rest can be checked directly. The easiest way is to use $[B\Spin(k), P^4BO] \simeq H^4(B\Spin(k); \Z)$ for $k \ge 3$, and then the equivalence is checked by restricting the representations to the maximal tori. We leave the details to the reader. 
\end{proof}

\begin{thm}[{Structures of $\RO(C_2)$-graded $\TMF$}]\label{thm_C2_twisted}
	~
	\begin{enumerate}
		\item 
		For each $2 \le k \le 6$, the level-rank duality morphism \eqref{eq_levelrank} is an isomorphism of $\TMF$-module spectra, 
		\begin{align}\label{eq_levelrank_isom}
			\cF'_{(C_2)_k} \colon  D(\TMF[\overline V_{\Spin(k)}]^{\Spin(k)}) \simeq \TMF[k\lambda]^{C_2} .  
		\end{align}
		\item Using the isomorphisms, we can further rewrite the equivalence \eqref{eq_levelrank_isom} in terms of $\TJF$ and $\TEJF$ as follows:
	\begin{align}
		\TMF[2\lambda]^{C_2} &\stackrel{\cF'_{(C_2)_2}}{\simeq} D(\TMF[\overline{V}_{\Spin(2)}]^{\Spin(2)}) \stackrel{\eqref{eq_thom_spin2}}{\simeq} D(\TJF_4)[8] \stackrel{\eqref{eq_twisted_dual}}{\simeq} \TJF_{-4}[9] \label{eq_2/8} \\
			\TMF[3\lambda]^{C_2}  &\stackrel{\cF'_{(C_2)_3}}{\simeq} D(\TMF[\overline{V}_{\Spin(3)}]^{\Spin(3)}) \stackrel{\eqref{eq_thom_spin3}}{\simeq} D(\TEJF_4)[8] \stackrel{\eqref{eq_twisted_dual}}{\simeq} \TEJF_{-8}[13] \label{eq_3/8} \\
			\TMF[4\lambda]^{C_2}  &\stackrel{\cF'_{(C_2)_4}}{\simeq} D(\TMF[\overline{V}_{\Spin(4)}]^{\Spin(4)}) \stackrel{\eqref{eq_thom_spin4}}{\simeq} D(\TEJF_2 \otimes_{\TMF} \TEJF_2)[8]\label{eq_4/8}  \\
		 & 	\qquad \qquad \stackrel{\cF'_{Sp(1)_1} \otimes \cF'_{Sp(1)_1}}{\simeq}  \TEJF_2 \otimes_\TMF \TEJF_2 \\
				\TMF[5\lambda]^{C_2}  &\stackrel{\cF'_{(C_2)_5}}{\simeq} D(\TMF[\overline{V}_{\Spin(5)}]^{\Spin(5)}) \stackrel{\eqref{eq_thom_spin5}}{\simeq} D(\TMF[\overline{V}_{Sp(2)}]^{Sp(2)}) \stackrel{\cF'_{Sp(1)_2}}{\simeq} \TEJF_4 \label{eq_5/8} \\
			\TMF[6\lambda]^{C_2}  &\stackrel{\cF'_{(C_2)_6}}{\simeq} D(\TMF[\overline{V}_{\Spin(6)}]^{\Spin(6)})\stackrel{\eqref{eq_thom_spin6}}{\simeq} D(\TMF[\overline{V}_{SU(4)}]^{SU(4)}) \stackrel{\cF'_{U(1)_4}}{\simeq}  \TJF_4. \label{eq_6/8}
	\end{align}
	In \eqref{eq_4/8}, \eqref{eq_5/8} and \eqref{eq_6/8}, we used the level-rank duality isomorphisms for $U/SU$ and $Sp/Sp$ proven in \cite[Section~6]{lin2024topologicalellipticgenerai}.
	See Figure~\ref{celldiag_TMFZ/2} for the cell diagrams.  
	\item The fiber sequence \eqref{seq_twistedZn_TJF} for $n=2$ and the isomorphism \eqref{eq_2/8} are related by the following commutative diagram:
	\begin{align}\label{diag_k=2}
		\xymatrix{
		\TMF[2\lambda-1]^{C_2} \ar[r] &	\TJF_{-3}[6] \ar[r]^-{\chi(\mu^2) \cdot }  & \TJF_1 \ar[rr]^-{\res_{U(1)}^{C_2}}  && \TMF[2\lambda]^{C_2} \\
		D(\TJF_4)[7] \ar[r]^-{\chi(\mu)\cdot} \ar[u]^-{\eqref{eq_2/8}}_-{\simeq}&	D(\TJF_3)[7] \ar[r] \ar[u]^-{\eqref{eq_duality_TJF}}_-{\simeq}& D(\TMF) = \TMF  \ar[rr]^-{D(\res_{U(1)}^e)}\ar[u]^-{\eqref{eq_TJF_1}}_-{\simeq} & & D(\TJF_4)[8]\ar[u]^-{\eqref{eq_2/8}}_-{\simeq}
	}
	\end{align}

	\item The following diagrams also commute:
	\begin{align}\label{diag_234}
		\xymatrix@C=7em{
		\TMF[2\lambda]^{C_2} \ar[r]^-{\chi(\lambda)\cdot } & \TMF[3\lambda]^{C_2} \ar[r]^-{\chi(\lambda)\cdot } & \TMF[4\lambda]^{C_2} \\
		D(\TJF_4)[8] \ar[r]^-{D(\res_{Sp(1)}^{U(1)})} \ar[u]^-{\eqref{eq_2/8}}_{\simeq}& D(\TEJF_4)[8]\ar[r]^-{D({\rm multi})} \ar[u]^-{\eqref{eq_3/8}}_{\simeq} & D(\TEJF_2 \otimes_\TMF \TEJF_2) [8] \ar[u]^-{\eqref{eq_4/8}}_{\simeq}
		}
	\end{align}

	\begin{align}\label{diag_456}
			\xymatrix@C=7em{
			\TMF[4\lambda]^{C_2} \ar[r]^-{\chi(\lambda)\cdot } & \TMF[5\lambda]^{C_2} \ar[r]^-{\chi(\lambda)\cdot } & \TMF[6\lambda]^{C_2} \\
			\TEJF_2 \otimes_\TMF \TEJF_2 \ar[r]^-{{\rm multi}} \ar[u]^-{\eqref{eq_4/8}}_{\simeq} 
			& \TEJF_4 \ar[r]^-{\res_{Sp(1)}^{U(1)}} \ar[u]^-{\eqref{eq_5/8}}_{\simeq} & \TJF_4 \ar[u]^-{\eqref{eq_6/8}}_{\simeq}
			}
	\end{align}
	Here, $\mathrm{multi}$ is the multiplication in the graded ring $\oplus_{k \in \Z} \TEJF_{2k}$. 

	\end{enumerate}
\end{thm}

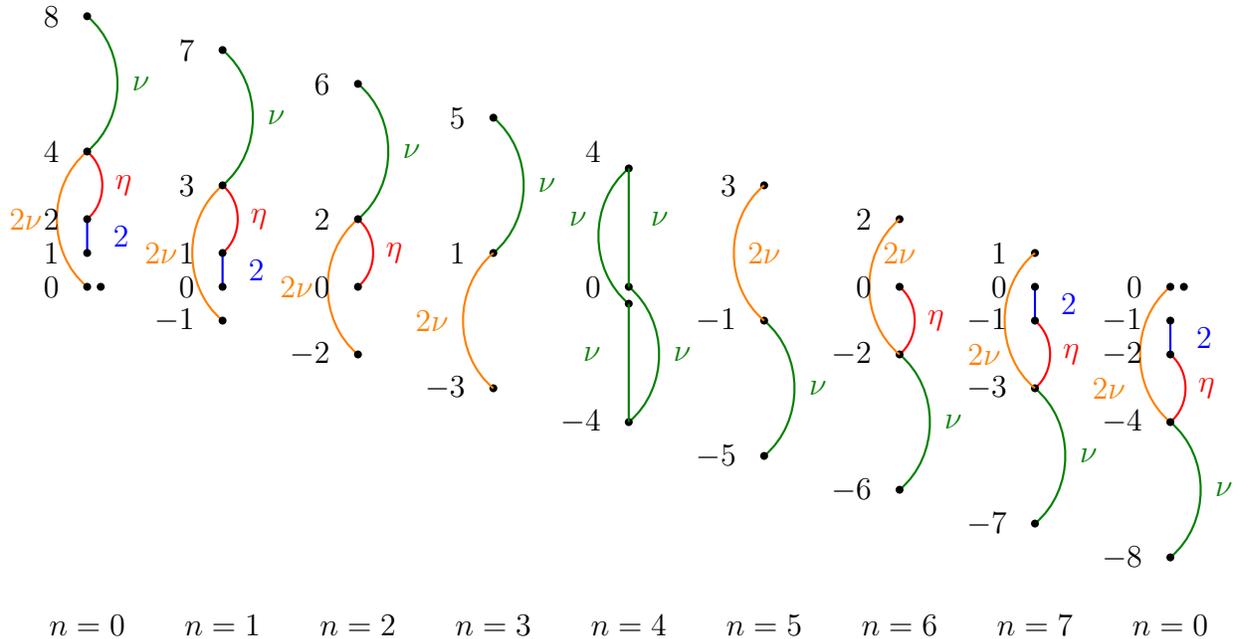
\begin{figure}[b]
	\centering
	\begin{tikzpicture}[scale=0.45]
		\begin{scope}[shift={(-8, 0)}]
			\begin{celldiagram}
				\nu{4}
				\eta{2}
				\two{1}
				\eightnu{0}
				\n{0} \n{1} \n{2} \n{4} \n{8}
				\node [circ] at (0.4, 0) {};
				\foreach \y in {0, 1, 2, 4, 8} {
					\node [left] at (-0.5, \y) {$\y$};
				}
			\end{celldiagram}
			\node [] at (0, -10) {$n=0$};
			\node [left, orange] at (-1, 2) {$2\nu$};
			\node [right, red] at (0.5, 3) {$\eta$}; 
			\node [right, green!50!black] at (1, 6) {$\nu$};
			\node [blue] at (1, 1.5) {$2$};
		\end{scope}
		
		\begin{scope}[shift={(-4, 0)}]
			\begin{celldiagram}
				\nu{3}
				\eta{1}
				\two{0}
				\eightnu{-1}
				\n{-1} \n{0} \n{1} \n{3} \n{7}
				
				\foreach \y in {-1, 0, 1, 3, 7} {
					\node [left] at (-0.5, \y) {$\y$};
				}
			\end{celldiagram}
			\node [] at (0, -10) {$n=1$};
				\node [left, orange] at (-1, 1) {$2\nu$};
			\node [right, red] at (0.5, 2) {$\eta$}; 
			\node [right, green!50!black] at (1, 5) {$\nu$};
			\node [blue] at (1, 0.5) {$2$};
		\end{scope}
		
		\begin{scope}[shift={(0, 0)}]
			\begin{celldiagram}
				\nu{2}
				\eta{0}
				\eightnu{-2}
				\n{-2} \n{0} \n{2} \n{6}
				
				\foreach \y in {-2, 0, 2, 6} {
					\node [left] at (-0.5, \y) {$\y$};
				}
			\end{celldiagram}
			\node [] at (0, -10) {$n=2$};
				\node [left, orange] at (-1, 0) {$2\nu$};
			\node [right, red] at (0.5, 1) {$\eta$}; 
			\node [right, green!50!black] at (1, 4) {$\nu$};
		\end{scope}
		
		\begin{scope}[shift={(4, 0)}]
			\begin{celldiagram}
				\nu{1}
				\n{-3} \n{1} \n{5}
				\eightnu{-3}

				\foreach \y in {-3, 1, 5} {
					\node [left] at (-0.5, \y) {$\y$};
				}
			\end{celldiagram}
			\node [] at (0, -10) {$n=3$};
				\node [left, orange] at (-1, -1) {$2\nu$};
			\node [right, green!50!black] at (1, 3) {$\nu$};
		\end{scope}
		
		\begin{scope}[shift={(8, 0)}]
			\begin{celldiagram}
				\nu{-4}
				\nu{-0.5}
				\n{-4} \n{0} \n{-0.5} \n{3.5} 
			\draw [thick, green!50!black] (0, -4) to (0, -0.5);
			\draw [thick, green!50!black] (0, 0) to (0, 3.5);
				
				\foreach \y in {0, -4, 4} {
					\node [left] at (-0.5, \y) {$\y$};
				}
			\end{celldiagram}
			\node [] at (0, -10) {$n=4 $};
				\node [left, green!50!black] at (-0.5, -2) {$\nu$};
			\node [green!50!black] at (-1.5, 2) {$\nu$};
			\node [right, green!50!black] at (1, -2) {$\nu$};
			\node [left, green!50!black] at (1.5, 2) {$\nu$};
		\end{scope}
		
		\begin{scope}[shift={(12, 0)}]
			\begin{celldiagram}
				\nu{-5}
				\n{-5} \n{-1} \n{3} 
				\eightnu{-1}
				\foreach \y in {-5, -1, 3} {
					\node [left] at (-0.5, \y) {$\y$};
				}
			\end{celldiagram}
			\node [] at (0, -10) {$n=5$};
				\node [orange] at (0, 1) {$2\nu$};
			\node [right, green!50!black] at (1, -3) {$\nu$};
		\end{scope}
		
		\begin{scope}[shift={(16, 0)}]
			\begin{celldiagram}
				\eta{-2}
				\nu{-6}
				\n{-6} \n{-2} \n{0} \n{2}
				\eightnu{-2}
				\foreach \y in {-6, -2, 0, 2} {
					\node [left] at (-0.5, \y) {$\y$};
				}
			\end{celldiagram}
			\node [] at (0, -10) {$n=6$};
				\node [orange] at (0, 1) {$2\nu$};
			\node [right, green!50!black] at (1, -4) {$\nu$};
			\node [right, red] at (0.5, -1) {$\eta$}; 
		\end{scope}
		
		\begin{scope}[shift={(20, 0)}]
			\begin{celldiagram}
				\two{-1}
				\eta{-3}
				\nu{-7}
				\n{-7} \n{-3} \n{-1} \n{0} \n{1}
				\eightnu{-3}
				\foreach \y in {1, 0, -1, -3, -7} {
					\node [left] at (-0.5, \y) {$\y$};
				}
			\end{celldiagram}
			\node [] at (0, -10) {$n=7$};
				\node [orange] at (-1.5, -2) {$2\nu$};
			\node [right, green!50!black] at (1, -5) {$\nu$};
			\node [right, red] at (0.5, -2) {$\eta$}; 
			\node [blue] at (1, -0.5) {$2$};
		\end{scope}
		
		\begin{scope}[shift={(24, 0)}]
			\begin{celldiagram}
				\two{-2}
				\eta{-4}
				\nu{-8}
				\eightnu{-4}
				\n{-8} \n{-4} \n{-2} \n{-1} \n{0} 
				\node [circ] at (0.4, 0) {};
				\foreach \y in {0, -1,-2,-4,-8} {
					\node [left] at (-0.5, \y) {$\y$};
				}
			\end{celldiagram}
			\node [] at (0, -10) {$n=0$};
			\node [right, green!50!black] at (1, -6) {$\nu$};
			\node [right, red] at (0.5, -3) {$\eta$}; 
			\node [left, orange] at (-1, -3) {$2\nu$};
			\node [blue] at (1, -1.5) {$2$};
		\end{scope}
	\end{tikzpicture}
	\caption{The cell diagrams of $\TMF[n\overline{\lambda}]^{C_2}$, $n \in \Z/8$.}\label{celldiag_TMFZ/2}
\end{figure}

The rest of this subsection is devoted to the proof of Theorem \ref{thm_C2_twisted}. Our proof is structured as follows. First, in Section \ref{sec_key_lem} we establish a key lemma concerning the map $\{c\} \colon \TJF_{-1}[6] \to \TJF_2$. Then in Section \ref{sec_proofstep1}, we prove statement (1) for $k=2$ and statement (3). In Section \ref{sec_proofstep2}, we complete the proof of (1) by induction on $k$. 
Finally, in Section \ref{sec_proofstep3}, we show statement (4).  

\begin{rem}\label{rem_sigma}
	We have seen in \cite[Section~2.3]{lin2024topologicalellipticgenerai} the subtlety regarding the genuine equivariant refinement of sigma orientations. 
	The groups of the form $C_2 \times \Spin(k)$ may not be sigma-oriented in general. However, we claim that the groups $C_2 \times \Spin(k)$ for $2 \le k \le 6$ are sigma-orientable. The full detail will appear in \cite{MeierYamashita}. 
	Note that by the isomorphisms \eqref{eq_spin_grp}, all the groups $\Spin(k)$ for $2 \le k \le 6$ are string-orientable. Also, we have shown that $C_2$ is sigma-oriented in Proposition~\ref{prop_sigma_Cp}. We will show that the product $C_2 \times \Spin(k)$ is also sigma-orientable from the Picard group argument of the product stacks, and the fact that the Elliptic cohomology functor applied to stacks of the form $\mathbf{B} \Sp(n)$ and $\mathbf{B} \SU(n)$ gives bundles of projective spaces over $\moduli^{\ori}$ \cite{GepnerMeierNEW}.
\end{rem}

\subsubsection{A key lemma on a multiplication in $\TJF$}\label{sec_key_lem}

Here we show a simple but important lemma for our analysis below. Recall we have an element $\{c\} \in \pi_6 \TJF_3$ which satisfies $e_{\JF}(\{c\}) = c = 2 + O(z)[[q]]$. 

\begin{lem}\label{lem_key}
	The homotopy class of the composition
	\begin{align}\label{eq_lem_key}
		\TMF[7] \stackrel{\eqref{eq_TJF-1}}{\simeq} \TJF_{-1}[6] \xrightarrow{\{c\}\cdot } \TJF_2
	\end{align}
coincides with $\zeta \in \pi_7 \TJF_2= \Z \zeta/12\zeta$. The element $\zeta$ is defined in Fact~\ref{fact_pi_TJF} \eqref{fact_pi7TJF2}. 
\end{lem}

\begin{proof}
We claim that the following diagram commutes:
\begin{align}\label{diag_key_lem}
	\xymatrix@C=7em{
\TMF[7] \ar[r]^-{\simeq}_-{\eqref{eq_TJF-1}} \ar[rd]_-{\nu \cdot}&	\TJF_{-1}[6] \ar[r]^-{\{c\}\cdot } \ar[d]^-{\res_{U(1)}^e} & \TJF_2 \ar[d]^-{\res_{U(1)}^e} \\
&	\TMF[4] \ar[r]^-{2 \cdot }& \TMF[4]
	}
\end{align}
To see right square, take the global section $c \in \pi_6\Gamma(\cE^\ori, \cO_{\cE^\ori}(3e))$ and regard it as a sheaf morphism $\Sigma^6 \cO_{\cE^\ori} \to \cO_{\cE^\ori}(3e)$. The middle arrow in \eqref{diag_P3TJF3} shows that, when restricted to the basepoint $e \colon \cM^\ori \to \cE^\ori$, the map $c$ induces the $2$-multiplication map between fibers $\TMF[6]$. Then the right square can be obtained by tensoring $\cO_{\cE^\ori}(-e)$.
To see the commutativity of the left triangle, we use the fact that the restriction-stabilization fiber sequence \eqref{eq_TJF_stabres} is compatible with the duality \eqref{eq_duality_TJF} \cite[(A.10)]{lin2024topologicalellipticgenerai}. 
In our case, it means that the stabilization-restriction fiber sequence
\begin{align}
	\TJF_{-2} \xrightarrow{\chi(\mu) \cdot } \TJF_{-1} \xrightarrow{\res_{U(1)}^e} \TMF[-2] 
\end{align}
is, after shifting the degree by one, $\TMF$-linear dual to the following stabilization-restriction fiber sequence which exhibits $\TJF_2$ as a cofiber of $\nu$ (see \eqref{eq_P1}, \eqref{eq_P2} and \eqref{eq_attaching_P}): 
\begin{align}
	\xymatrix{
		\TJF_{2} \ar[d]^-{\simeq}& \ar[l]_-{\chi(\mu) \cdot} \ar[d]^-{\simeq} \TJF_1  & \ar[l]_-{\tr_{e}^{U(1)}} \TMF[3] \ar@{=}[d] \\
		\TMF/\nu & \TMF \ar[l] & \TMF[3] \ar[l]_-{\cdot \nu}
	}
\end{align}
This implies the commutativity of the left triangle of \eqref{diag_key_lem}, and furthermore, that the homotopy class of the composition \eqref{eq_lem_key} is an element in $\pi_7 \TJF_2$ that restricts to $2\nu \in \pi_3 \TMF$---the definition of $\zeta$.
\end{proof}

\begin{rem}[A direct proof for weak statements]\label{rem_C2_weaker}
	At this point, we can prove the identification as a $\TMF$-module
	\begin{align}
		\TMF[6\lambda]^{C_2} \simeq \TJF_4, \quad &\TMF[2\lambda]^{C_2} \simeq D(\TJF_4)[8], \label{eq_rem_twistedC2_1} \\
		\TMF[5\lambda]^{C_2} \simeq \TEJF_4 , \quad &\TMF[3\lambda]^{C_2} \simeq D(\TEJF_4)[8]. \label{eq_rem_twistedC2_2}
	\end{align}
	Indeed, the fiber sequence
	\begin{align}
		\TJF_{-1} \xrightarrow{\chi(\mu^2) = \chi(\mu) \circ\{c\}} \TJF_3 \xrightarrow{\res_{U(1)}^{C_2}} \TMF[6\lambda]^{C_2}
	\end{align}
	from Proposition \ref{prop_twisted}, together with Lemma \ref{lem_key} and the equality $\chi(\mu) \cdot \zeta = \gamma$ in $\pi_7 \TJF_3$ (Fact \ref{fact_pi_TJF} \eqref{fact_pi7TJF3}) implies that we have an isomorphism
	\begin{align}
		\TMF[6\lambda]^{C_2} \simeq  \TJF_4, 
	\end{align}
	where we have used \eqref{fact_TJF3toTJF4}. 
	By the duality $\TMF[k\lambda]^{C_2}[8] \simeq D(\TMF[(8-k)\lambda])$, we get the second isomorphism. 
	Similarly, for the $5\lambda$-twist, we use the fiber sequence
	\begin{align}
		\TJF_{-1} \xrightarrow{\{c\}} \TJF_2 \xrightarrow{\chi(\lambda) \circ \res_{U(1)}^{C_2}} \TMF[6\lambda]^{C_2}
	\end{align}
	from Propositions \ref{prop_oddtwisted}, \ref{fact_zeta_TEJF}, and Lemma \ref{lem_key} to get
	\begin{align}
		\TMF[5\lambda]^{C_2}\simeq \TEJF_4.
	\end{align} 
	The result on the $3\lambda$-twist follows from the dual. 
\end{rem}

\subsubsection{Proof of Theorem \ref{thm_C2_twisted}, Step 1 : Showing (1) for $k=2$ and (3)}\label{sec_proofstep1}
First, let us consider the $k=2$ case. 
We already know that we have a diagram
\begin{align}\label{diag_proofstep1}
		\xymatrix@C=4em{
		\TJF_{-3}[6] \ar[rr]^-{\chi(\mu^2) \cdot } && \TJF_1 \ar[rr]^-{\res_{U(1)}^{C_2}} && \TMF[2\lambda]^{C_2} \\
		D(\TJF_3)[7] \ar@{=}[d] \ar[rr]^-{D(\chi(\mu^2)\cdot)} \ar[u]^-{\eqref{eq_duality_TJF}}_-{\simeq} &&D(\TJF_{-1})[1] \ar[u]^-{\eqref{eq_duality_TJF}}_-{\simeq} && D(\TMF[\overline V_{\Spin(2)}]^{\Spin(2)}) \ar[u]^-{\cF'_{(C_2)_2}}\\
			D(\TJF_3)[7] \ar[rr]^-{D(\gamma)} && D(\TMF) = \TMF  \ar[rr]^-{D(\res_{U(1)}^e)} \ar[u]^-{\eqref{eq_TJF-1}}_-{\simeq }  \ar@/_40pt/[uu]_-{\chi(\mu)} \ar[rruu]^-{\chi(2\lambda)}& & D(\TJF_4)[8] \ar[u]^-{\simeq}_-{\eqref{eq_spin_grp}}
	}
\end{align}
whose top and bottom rows are fiber sequences, and whose vertical arrows in the left and middle columns are isomorphisms. So we are left to show that this diagram commutes; indeed, it would imply the theorem's statement (3) as well as (1) for $k=2$. 

The top left square commutes since the duality in $U(1)$-equivariant $\TMF$ is compatible with the graded ring structure on $\oplus_m \TJF_m$.
The middle triangle commutes by the definition of the isomorphism \eqref{eq_TJF-1}. 
The right upper triangle commutes since $$2\lambda \simeq \res_{U(1)}^{C_2}(\mu). $$ 
The bottom left square commutes by Lemma \ref{lem_key} and the fact that $$\gamma = \chi(\mu) \cdot \zeta, \mbox{and }  \chi(\mu^2) = \chi(\mu) \cdot\{c\} . $$

Note that, by the naturality of Spanier-Whitehead duality, the bottom-right composition of the right square in \eqref{diag_proofstep1} is equivalent to the composition
\begin{align}\label{eq_proofstep1_coev}
	\TMF \xrightarrow{\cF_{(C_2)_2}}\TMF[2\lambda]^{C_2} \otimes_\TMF \TMF[\overline V_{\Spin(2)}]^{\Spin(2)} \xrightarrow{\id \otimes \res_{\Spin(2)}^{e}} \TMF[2\lambda]^{C_2}. 
\end{align}
Now, consider the diagram
 \begin{align}\label{diag_proofstep1_coev}
	\xymatrix@C=5em{
	\TMF \ar[r]|-{\chi(V_{\phi_2})}  \ar[rd]_-{\chi(\res_{C_2 \times \Spin(2)}^{C_2} (V_{\phi_2}))} \ar@/^30pt/[rr]^-{\cF_{(C_2)_2}} & \TMF[V_{\phi_2}]^{C_2 \times \Spin(2)} \ar[r]_-{\simeq}^-{\sigma(\Theta_2, \fraks)} \ar[d]^-{\res_{C_2 \times \Spin(2)}^{C_2}}& \TMF[2\lambda]^{C_2} \otimes_\TMF \TMF[\overline V_{\Spin(2)}]^{\Spin(2)} \ar[d]^-{\id \otimes \res_{\Spin(2)}^e} \\
& \TMF[\res^{C_2}_{C_2 \times \Spin(2)} V_{\phi_2}]^{C_2} \ar[r]^-{\simeq}_-{\eqref{eq_proofstep1_rep}}& \TMF[2\lambda]^{C_2}
	}
\end{align}
Here, the bottom horizontal equivalence is given by the isomorphism, 
\begin{align}\label{eq_proofstep1_rep}
	\res_{C_2 \times \Spin(2)}^{C_2} (V_{\phi_2}) = \res_{C_2 \times \Spin(2)}^{C_2} (\lambda \otimes_\R V_{\Spin(2)}) \simeq 2\lambda \quad \mbox{in } \RO(C_2).  
\end{align}
The square commutes because of the definition of the string structure $\fraks$ on $V_{\phi_2}$ in \cite[Proposition 4.26]{lin2024topologicalellipticgenerai}. 
The lower left triangle commutes by the functoriality of the equivariant Euler class. Because the composition of the bottom arrows is $\chi(2\lambda)$, this completes the proof of the commutativity of the diagram \eqref{diag_proofstep1}.

\subsubsection{Proof of Theorem \ref{thm_C2_twisted}, Step 2 : Showing (1) for $3 \le k \le 6$} \label{sec_proofstep2}

We prove it inductively on $k$. This part of the proof is analogous to the proof of the $U/SU$ and $Sp/Sp$ level-rank duality statements in \cite[Section~6]{lin2024topologicalellipticgenerai}.

As a special case of the stabilization-restriction fiber sequences \cite[Proposition~4.45]{lin2024topologicalellipticgenerai}, we have the following fiber sequence for any $k$:
\begin{align}\label{eq_stabres_spin}
	\TMF[-k]^{\Spin(k)} \xrightarrow{\chi(V_{\Spin(k)}) \cdot } \TMF[\overline{V}_{\Spin(k)}]^{\Spin(k)} \xrightarrow{\res_{\Spin(k)}^{\Spin(k-1)}}\TMF[\overline V_{\Spin(k-1)}]^{\Spin(k-1)}.
\end{align}
Moreover, we claim that the following restriction map is an isomorphism for $3 \le k \le 6$:\footnote{This condition on the range of $k$ is essential: for $k=2$ we have $\Spin(2) \simeq U(1)$ and the restriction map is not an isomorphism.}
\begin{align}
	\res_{\Spin(k)}^e \colon \TMF^{\Spin(k) } \simeq \TMF \mbox{ for } 3 \le k \le 6. 
\end{align}
Indeed, it follows from the list of identifications \eqref{eq_spin_grp} of the Spin groups in this range with products of the $SU$ and $Sp$ groups. The corresponding restriction maps are isomorphisms for those series of groups (see \cite{GepnerMeierNEW} and \cite[Fact~6.5]{lin2024topologicalellipticgenerai})
\begin{align}
	\res_{G}^e \colon \TMF^{G } \simeq \TMF \quad \mbox{ for } G = SU(n), \ G=Sp(n).
\end{align}

The level-rank duality morphisms \eqref{eq_levelrank} are compatible with the stabilization-restriction sequences by \cite[Proposition~4.97]{lin2024topologicalellipticgenerai}. In our case, for $3 \le k \le 6$, it means that the following diagram commutes:

\begin{align}\label{diag_proofstep2}
	\xymatrix@C=6em{
		\TMF[(k-1)\lambda]^{C_2} \ar[r]^-{\chi(\lambda) \cdot } & \TMF[k\lambda]^{C_2} \ar[r]^-{\res_{C_2}^e} & \TMF[k] \ar@{=}[d]\\
		D(\TMF[\overline V_{\Spin(k-1)}]^{\Spin(k-1)}) \ar[r]^-{D(\res_{\Spin(k)}^{\Spin(k-1)})} \ar[u]^-{\cF'_{(C_2)_{k-1}}}&	D(\TMF[\overline V_{\Spin(k)}]^{\Spin(k)}) \ar[r]^-{D(\chi(V_{\Spin(k)}))} \ar[u]^-{\cF'_{(C_2)_{k}}}  & D(\TMF[-k])
	}
\end{align}
here, the upper row is the stabilization-restriction fiber sequence for $C_2$, and the bottom row is the dual of \ref{eq_stabres_spin}. This completes the inductive proof of Theorem \ref{thm_C2_twisted} (1). 

\subsubsection{Proof of Theorem \ref{thm_C2_twisted}, Step 3: The proof of (4)}\label{sec_proofstep3}
By the commutativity of the left square of \eqref{diag_proofstep2}, the multiplication of $\chi(\lambda)$ is identified with the dual of the restriction map along $\Spin(k-1) \hookrightarrow \Spin(k)$. 
Recall that the group isomorphisms \eqref{eq_spin_grp} identify the inclusion of Spin groups as
\begin{align}
	U(1) \hookrightarrow Sp(1) \stackrel{\rm \id \times \id}{\hookrightarrow} Sp(1) \times Sp(1)  \stackrel{\rm diag}{\hookrightarrow} Sp(2) \hookrightarrow SU(4). 
\end{align}
The commutativity of \eqref{diag_234} follows from the first two identifications. 
For \eqref{diag_456}, we should further show the commutativity of the following diagram: 
\begin{align}\label{diag_proofstep3}
	\xymatrix@C=5em{
	D(\TEJF_2 \otimes_\TMF \TEJF_2 )[8] \ar[r]^-{D(\res_{\rm diag})} \ar[d]_{\simeq}^-{\cF'_{Sp(1)_1} \otimes \cF'_{Sp(1)_1}}& D(\TMF[\overline V_{Sp(2)}]^{Sp(2)}) \ar[r]^-{D(\res_{SU(4)}^{Sp(2)})} \ar[d]_-{\simeq}^-{\cF'_{Sp(1)_2}} & D(\TMF[\overline V_{SU(4)}]^{SU(4)})  \ar[d]^-{\cF'_{U(1)_4}}_-{\simeq}  \\
	\TEJF_2 \otimes_\TMF \TEJF_2 \ar[r]^-{\rm{multi}} & \TEJF_4 \ar[r]^-{\res_{Sp(1)}^{U(1)}} & \TJF_4
}
\end{align}
Here, the vertical arrows consist of level-rank duality isomorphisms for $U/SU$ and $Sp/Sp$ in \cite{lin2024topologicalellipticgenerai}. The commutativity of the diagram \eqref{diag_proofstep3} follows by the functoriality statement in \cite[Proposition~3.65]{lin2024topologicalellipticgenerai}.\footnote{The right square is a special case of \cite[(4.30)]{lin2024topologicalellipticgenerai}. For the left square, we use the fact that the restriction along the group homomorphism (here $a,b,c,d,x$ label copies of $Sp(1)$)
$$\id \times (\id, \id) \colon \left( Sp(1)_a \times Sp(1)_b \right) \times Sp(1)_x \to \left( Sp(1)_a \times Sp(1)_b	\right) \times \left( Sp(1)_c \times Sp(1)_d	\right)$$
of the representation 
\begin{align}
	\left( V_{Sp(1)_a} \otimes_\bH V_{Sp(1)_c}^*\right)  \oplus \left(V_{Sp(1)_b} \otimes_\bH V_{Sp(1)_d}^* \right)
\end{align}
is equivalent to the restriction along the group homomorphism
$$\mathrm{diag} \times \id \colon \left( Sp(1)_a \times Sp(1)_b \right) \times Sp(1)_x \to Sp(2) \times Sp(1)_x$$
of the representation 
\begin{align}
	V_{Sp(2)} \otimes_\bH V_{Sp(1)_x}^*.
\end{align}
}
This finishes the proof of (4) and completes the proof of Theorem~\ref{thm_C2_twisted}. 

\section{Application 2 : $3$-local $C_3$-equivariant $\TMF$}\label{sec_C3}

We apply our general strategy in Section \ref{sec_general} to $n=3$ to study the $3$-local structure of $\TMF^{C_3}$. 
The structure of $\TMF^{C_n}$ without any $\RO(C_n)$ twist has been investigated extensively, and among the prime-order cyclic groups, the $3$-local structure of $\TMF^{C_3}$ was the remaining open case. 
In this section, we resolve this final case by explicitly determining the $\pi_* \TMF$-module structure of $\pi_* \TMF^{C_3}$. 
{\it Throughout this section, all spectra are implicitly $3$-localized. }

The strategy is to apply the result of Section \ref{subsec_twisted} for {\it twisted} cases. 
Consider the equivalence
\begin{align}\label{eq_C3twist_trivial}
	\TMF[-3\rho_3]^{C_3} \simeq \TMF[-6]^{C_3}, 
\end{align}
by Proposition \ref{prop_sigma_Cp}.
Thus, Proposition \ref{prop_twisted}, applied to $k=-3$ and $n=3$, gives the fiber sequence
\begin{align}\label{eq_TMFC3_twisted_fibseq}
	\TJF_6[-12] \xrightarrow{\res_{U(1)}^{C_3}} \TMF^{C_3} \xrightarrow{\tr_{C_3}^{U(1)}} \TJF_{-3}[5] \xrightarrow{\chi(\mu^3)} \TJF_6[-11].
\end{align}

\begin{rem}
	The reason why we do not use Proposition \ref{prop_general_nontwist}
	 is that the resulting fiber sequence
	\begin{align}
	\TJF_0[-2] \to \TJF_9[-18]	\to \TMF^{C_3} \to \TJF_0[-1],
	\end{align}
	is not split.
\end{rem}

\begin{lem}\label{lem_TJF-3_6}
	We have the following isomorphisms of $\TMF$-modules:
	\begin{align}
		\TJF_{-3} &\simeq \TMF[-5] \oplus \TMF/\alpha [-3], \\
		\TJF_{6} &\simeq \TMF_1(2) \oplus \TMF/\alpha [6 ] \oplus \TMF[12]. 
	\end{align}
\end{lem}
\begin{proof}
	The decomposition follows from Proposition \ref{prop_structure_TJF_2inverted} and the duality \eqref{eq_duality_TJF}.
\end{proof}

\begin{thm}[{$3$-local $\TMF$-module structure of $\TMF^{C_3}$}]\label{thm_C3_3loc_main}
	The fiber sequence \eqref{eq_TMFC3_twisted_fibseq} is split at $\TMF^{C_3}$. Thus, we obtain the following decomposition of $\TMF^{C_3}$ as a $\TMF$-module:
	\begin{align}
		\TMF^{C_3} &\simeq 	\TJF_6[-12] \oplus \TJF_{-3}[5] \\
		&\simeq \TMF_1(2)[-12] \oplus \TMF/\alpha[-6 ] \oplus \TMF \oplus \TMF \oplus \TMF/\alpha[2] \\
		&\simeq \TMF \otimes \left(S^{-4} \cup_\alpha S^0 \cup_\alpha S^4 \oplus S^{-6} \cup_\alpha S^{-2} \oplus S^0 \oplus S^0 \oplus S^2 \cup_\alpha S^6 \right)  , 
	\end{align}
	where the second equivalence used Lemma \ref{lem_TJF-3_6}, and the third equivalence used \eqref{eq_TMF_1(2)_cell} and \eqref{eq_TMF_1(2)_8periodic}. 
	The corresponding cell diagram is in Figure \ref{fig:cell_TMFC3}. 
\end{thm}

\begin{proof}
	It suffices to show the $\TMF$-module morphism
	\begin{align}
	\chi(\mu^3) \colon \TJF_{-3}[5] \to \TJF_6[-11]
	\end{align}
	is null-homotopic. We establish the stronger statement that 
	\begin{align}\label{eq_proof_C3_main}
	\left[\TJF_{-3}[5],  \TJF_6[-11]\right] = 0, 
	\end{align}
	where here $[-, -]$ denotes the group of homotopy classes of $\TMF$-module morphisms. 

Using Lemma \ref{lem_TJF-3_6}, we rewrite the hom set as
\begin{align}
		\nonumber \left[\TJF_{-3}[5], \TJF_6[-11]\right] & \simeq \pi_{3} \TMF_1(2) \oplus \pi_5 \TJF_2 \oplus \pi_{-1} \TMF \\
		& \oplus \left[ \TMF/\alpha, \TMF_1(2)[-5]\right] \oplus \left[ \TMF/\alpha, \TJF_2[-7]\right] \oplus \left[ \TMF/\alpha, \TMF[-1]\right] . 
\end{align}
Each term vanishes for the following reasons:
\begin{itemize}
    \item $\pi_3 \TMF_1(2)=0$ since the homotopy groups of $\TMF_1(2)$ are concentrated in even degrees.
    \item $\pi_5 \TJF_2 \simeq \pi_5 \TMF/\alpha = 0$, and $\pi_{-1}\TMF=0$.
\end{itemize}
For the last three factors, we invoke a long exact sequence for a $\TMF$-module spectrum $M$ induced by the multiplication by $\alpha$:
\begin{align}
\cdots \to \pi_{1-k}M \xrightarrow{\alpha \cdot} \pi_{4-k} M \to \left[\TJF_2, M[k]\right] \to \pi_{-k} M \xrightarrow{\alpha \cdot} \pi_{3-k} M \to \cdots. 
\end{align}
In particular, 
\begin{itemize}
	\item $ \left[ \TMF/\alpha, \TMF_1(2)[-5]\right] =0$ because the homotopy groups of $\TMF_1(2)$ vanish in odd degrees;
	\item $ \left[ \TMF/\alpha, \TJF_2[-7]\right]=0$ because by the diagram in Figure~\ref{fig:TJF2}, $\pi_{11} \TJF_2 =0$ and $\pi_7 \TJF_2 \simeq \Z/3$ whose generator $\{\alpha \frac{b}{3}\}$ does not vanish under the $\alpha$-multiplication (indeed, we have an exotic extension $\alpha \cdot \{\alpha \frac{b}{3}\} = \beta\{a^2\} \neq 0$ in $\pi_{10}\TJF_2$);
	\item $\left[ \TMF/\alpha, \TMF[-1]\right] =0$ as $\pi_5\TMF$ and $\pi_1 \TMF$ vanish.
\end{itemize}
These observations complete the proof of \eqref{eq_proof_C3_main} and Theorem \ref{thm_C3_3loc_main}. 
\end{proof}

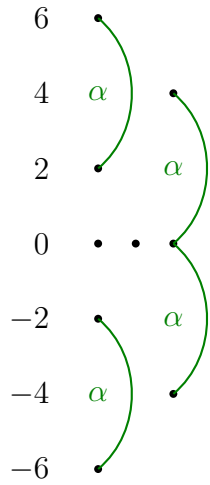
\begin{figure}[h]
	\centering
	\begin{tikzpicture}[scale=0.5]
		\begin{celldiagram}
			\n{-6}  \n{-2}   \n{2}  \n{6}
				\node [circ] at (0, 0) {};
					\node [circ] at (1, 0) {};
					\node [circ] at (2, -4) {};
					\node [circ] at (2, 0) {};
					\node [circ] at (2, 4) {};
					\draw [green!50!black, thick] (0, -6) to [bend \getside{nu}=50] (0, -2); 
					\draw [green!50!black, thick] (0, 2) to [bend \getside{nu}=50] (0, 6); 
					\draw [green!50!black, thick] (2, -4) to [bend \getside{nu}=50] (2, 0); 
					\draw [green!50!black, thick] (2, 0) to [bend \getside{nu}=50] (2, 4); 
			\foreach \y in {-6, -4, -2, 0, 2, 4, 6} {
				\node [left] at (-1, \y) {$\y$};
			}
		\end{celldiagram}
		\node [ green!50!black] at (0, -4) {$\alpha$};
		\node [ green!50!black] at (0, 4) {$\alpha$};
		\node [ green!50!black] at (2, -2) {$\alpha$};
		\node [ green!50!black] at (2, 2) {$\alpha$};
	\end{tikzpicture}
	\caption{The cell diagram of $\TMF^{C_3}_{(3)}$.}\label{fig:cell_TMFC3}
\end{figure}

\begin{figure}[h]
    \centering
    \includegraphics[width=1\linewidth]{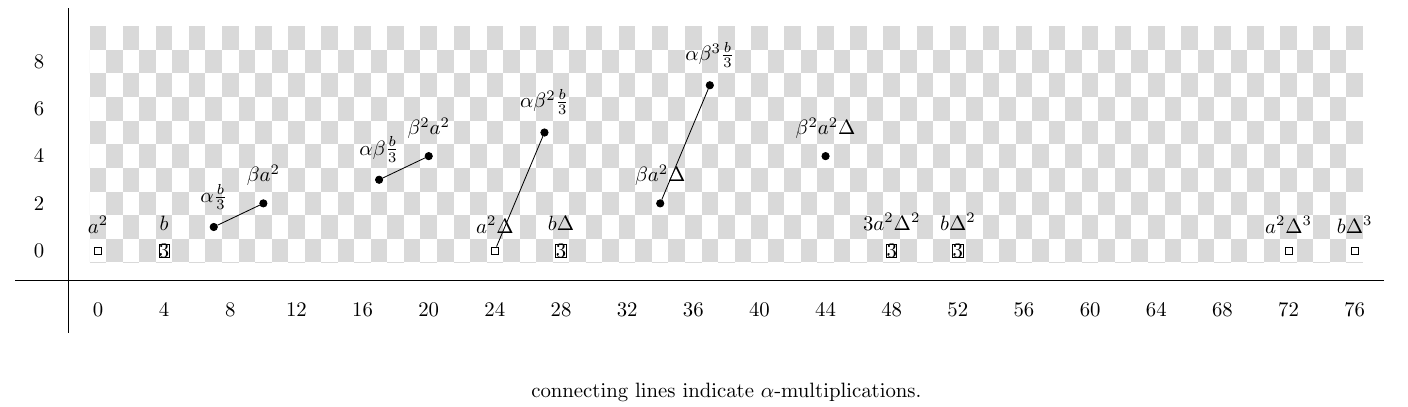}
    \caption{The $E_\infty$-page of DSS for the 3-local $\TJF_2$.\protect\footnotemark}
    \label{fig:TJF2}
\end{figure}

\footnotetext{Connecting lines indicate $\alpha$-multiplications and dotted lines indicate $\beta$-multiplications.}

\clearpage

\appendix

\section{On $\TJF_*$ after inverting $2$}\label{appendix_TJF2inverted}

This section discusses the structure of $\TJF_k$ after inverting the prime $2$, such as after $3$-localization. We state the main result: 
\begin{prop}[{Structure of $\TJF_m$ after inverting $2$}]\label{prop_structure_TJF_2inverted}
	Upon inverting $2$, the structure of the $\TMF$-module $\TJF_m$ is as follows.
	\begin{enumerate}
	\item For $m$ ranging from $1$ to $3$, the $\TMF$-modules are isomorphic as given:
	\begin{align}
		\TJF_1 &\cong \TMF , \\
		\TJF_2 &\cong \TMF/\alpha, \\
		\TJF_3 &\cong \TJF_2 \oplus \TMF[6].
	\end{align}
	\item Define $m' := \lfloor (m-1)/3 \rfloor$. For $m \geq 4$, we have the following isomorphism of $\TMF$-modules:
	\begin{align}\label{eq_deco_TJFm}
		\TJF_m\cong \TJF_{m-3m'}[6m'] \oplus \bigoplus_{i=0}^{m'-1} \TMF_1(2) [6i] .
	\end{align}
	\end{enumerate}
\end{prop}

(1) follows directly from the cell structure $\TJF_m \simeq \TMF \otimes P^m$. 
Before we proceed, we remark on the following fact about $\TMF_1(2)$:

\begin{rem}[Cell structures of $\TMF_1(2)$]\label{rem_TMF_1(2)}
	When $2$ is invertible, there are $\TMF$-module isomorphisms given by:
		\begin{align}\label{eq_TMF_1(2)_cell}
		\TMF_1(2) \simeq \TMF \otimes \left( S^0 \cup_\alpha S^4 \cup_\alpha S^8 \right)  \simeq \TMF \otimes \left( S^0 \cup_\alpha S^4 \cup_{2\alpha}  S^8\right) \simeq \TMF \otimes \left( S^0 \cup_{2\alpha} S^4 \cup_{\alpha}  S^8\right),
	\end{align}
	where the initial equivalence is a well-established result (e.g., see \cite[Theorem 13.4]{BrunerRognes}), and the subsequent equivalences arise from the automorphism $-1$ on the top and bottom cells. Furthermore, we find $\pi_* \TMF_1(2) \simeq \Z[\frac12][a_2, a_4, \Delta^{-1}]$, where $\Delta = a_4^2(a_2^2 - a_4)$. This indicates that $a_4$ has an inverse \(\frac{a_4(a_2^2-a_4)}{\Delta}\) within $\pi_{-8}\TMF_1(2)$, making $\TMF_1(2)$ an $8$-periodic $\TMF$-module:
	\begin{align}\label{eq_TMF_1(2)_8periodic}
		\TMF_1(2) \simeq \TMF_1(2)[8].
	\end{align}
	It is important to note that this isomorphism is not canonical; indeed, one could alternatively select $a_2^2-a_4$ as the periodicity element.
\end{rem}

The remainder of this section is devoted to the proof of Proposition \ref{prop_structure_TJF_2inverted} (2). 

\begin{lem}[{Decomposition of $\TJF_4$}]\label{lem_TJF4} 
    Suppose once more that $2$ is invertible.
	\begin{enumerate}
		\item There exists a unique element $\{c\} \in \pi_6 \TJF_3$ such that its Jacobi form image under $e_\JF$ equals $c := \phi_{0, 3/2} \in \pi_6 \JF_3$. 
		\item Multiplication by $\{c\} \in \pi_6 \TJF_4$ yields a split fiber sequence:
		\begin{align}\label{eq_fibseq_TJF4}
		 \TJF_1[6] \xrightarrow{\{c\} \cdot } \TJF_4 \to \TMF_1(2).
		\end{align}
	\end{enumerate}
\end{lem}
\begin{proof}
	Use $\TJF_4 \simeq \TMF \otimes P^4$ and the fact that $P^4$ decomposes as $P^4  \simeq S^0 \cup_{\alpha}S^4 \cup_{2\alpha}S^8 \oplus S^6$ after inverting $2$, and compare it with $\TMF_1(2) \simeq \TMF \otimes S^0 \cup_{\alpha} S^4 \cup_{2\alpha} S^8$ in \eqref{eq_TMF_1(2)_cell}. 
\end{proof}

\begin{defn}\label{def_zeta}
	Inverting $2$, we choose a splitting of the left arrow in \eqref{eq_fibseq_TJF4} and denote it by
	\begin{align}
		\psi \colon \TMF_1(2) \hookrightarrow \TJF_4. 
	\end{align}
\end{defn}

\begin{rem}
	We can take $\psi$ to be the restriction map $\res_{Sp(1)}^{U(1)} \colon \mathrm{TEJF}_{4} \to \TJF_4$ \cite[Appendix B]{lin2024topologicalellipticgenerai}. 
\end{rem}

We can now state a more precise version of Proposition \ref{prop_structure_TJF_2inverted} (2): 

\begin{prop}\label{prop_precise_structure_TJF}
	When $2$ is inverted, for any integer $m \ge 4$, $\TJF_m$ admits the following decomposition. Let $m' = \lfloor (m-1)/3 \rfloor$. Consider the mapping
	\begin{align}
		\left(\{c\}^{m'}\cdot, \ \bigoplus_{i=0}^{m'-1} \left(\{a\}^{m-4-3i}\{c\}^i \cdot \right) \circ \psi \  \right)  \colon \TJF_{m-3m'}[6m'] \oplus \bigoplus_{i=0}^{m'-1} \TMF_1(2)[6i] \to \TJF_m. 
	\end{align}
	This map is an equivalence of $\TMF$-modules. The components of this map include the inclusion $\psi$ as defined in Definition \ref{def_zeta}, together with the multiplications by $\{c\}\in \pi_6 \TJF_3$ and the element $\{a\} \in \pi_0 \TJF_1$, whose Jacobi form image is given by $\phi_{-1, \frac12}$.
\end{prop}

\begin{proof}
	We show that the map
	\begin{align}\label{eq_dec_induct}
		\left(\{c\} \cdot, \ \left(\{a\}^{k-4} \cdot \right) \circ \psi \right) \colon \TJF_{k-3}[6] \oplus \TMF_1(2) \to \TJF_k
	\end{align}
	is an equivalence for all $k \ge 4$. Then, the proposition follows by repeatedly applying this claim to $k=m, m-3, \cdots, m-3(m'-1)$. 
	
	We prove the above claim by induction on $k$. The case $k =4$ is addressed by Lemma \ref{lem_TJF4}. Consider the following commutative diagram of cofiber sequences:
	\begin{align}
		\begin{xymatrix}{
				\TJF_{k-3}[6] \oplus \TMF_1(2) \ar[rr]^-{\left(\{a\} \cdot, \ \id\right) } \ar[d]_{\left(\{c\} \cdot, \ \left(\{a\}^{k-4} \cdot \right) \circ \psi \right) } && \TJF_{k-2}[6] \oplus \TMF_1(2)\ar[rr]^-{\left(\res_{U(1)}^e, \ 0 \right)} \ar[d]_{\left(\{c\} \cdot, \ \left(\{a\}^{k-3} \cdot\right) \circ \psi \right) } & &\TMF[2k+2] \ar[d]_{\simeq}^-{\res_{U(1)}^e(\{c\})=2 \cdot} \\
				\TJF_{k} \ar[rr]^-{\{a\}\cdot} && \TJF_{k+1} \ar[rr]^-{\res_{U(1)}^e} && \TMF[2k+2] \\
			}
		\end{xymatrix}
	\end{align}
	The rightmost vertical arrow is identified as multiplication by $2$ since $c(z)=\phi_{0, 3/2}(z) = 2 + O(z)$, and is therefore an equivalence. 
	By the commutativity of the diagram, we see that the equivalence of the left vertical arrow implies the equivalence of the middle arrow. This completes the proof of the claim that \eqref{eq_dec_induct} is an equivalence and concludes the proof of Proposition \ref{prop_precise_structure_TJF} and of Proposition \ref{prop_structure_TJF_2inverted}. 
\end{proof}

\section{$2$-local Descent Spectral Sequence Charts}\label{appendix_TJF2}
Here, we show diagrams of the descent spectral sequences (DSS) for $\TEJF_4$ and $\TJF_4$, which are part of $\RO(C_2)$-graded $\TMF$. These charts are drawn in Adams: elements in $E_2^{s,t} \simeq \Ext^{s,t}$ are plotted in coordinates $(t-s, s)$, and differentials have degree $d_r \colon E_r^{s,t} \to E_r^{s+r, t+r-1}$. We adopt the following conventions:
\begin{itemize}
    \item A dot ``$\bullet$'' represents a generator of the cyclic group $\bZ / 2$.
    \item A circle around an element denotes a nontrivial $\bZ / 2$-extension of the group represented by that element.
    \item A square ``$\square$'' denotes a factor of $\bZ_{(2)}$.
    \item A number $n$ in the square indicates $n \bZ_{(2)}$-summand.
    \item A diamond ``$\diamond$'' in the $E_2$-page shows the repeated $\eta$-multiplication, meaning that all $\eta$-multiples from that element survive.
    \item A vertical line denotes multiplication by $2$.
    \item A non-vertical line of positive slope denotes multiplication with $\eta$ or $\nu$. Note that not all exotic $\eta$ and $\nu$ extensions are shown in the $E_\infty$-page, so there might be non-trivial extensions in the homotopy groups.
    \item A non-vertical arrow of negative slope denotes a differential.
    \item The $E_2$-term of DSS was computed using the cellular filtration, and we encode the origin of each class by color: the color of an element in the $E_2$-page indicates the $\TMF$-cell of $\TJF$ or $\TEJF$ from which it originates at the $E_1$-level. The $E_1$-page is omitted from the figures, as the relevant information is retained in the coloring. Specifically, classes from 0-cells are colored black, those from 4-cells are brown, 6-cells red, and 8-cells orange.
    \item The color of the multiplication lines indicates the image of the generator: for example, in Figure \ref{fig:TEJF4diff1}, $\nu$ times the generator in bidegree $(11,1)$ is the orange class in bidegree $(14,2)$, and therefore the $\nu$-multiplication induces an isomorphism between $E_{2}^{1, 12} \simeq \bZ / 4 \to E_2^{2,16} \simeq \bZ / 4$.
\end{itemize}

Figures \ref{fig:TEJF4diff1} to \ref{fig:TEJF4diff4} show $2$-local DSS for $\TEJF_4$. Its $E_2$-term is computed in \cite{BauerComputationTJF}, and differentials can be deduced similarly to the DSS for $\TJF$. Figures \ref{fig:piTJF4first} and \ref{fig:piTJF4second} show the $E_\infty$-page of DSS for $\TJF_4$. Its $E_2$-term and differentials are computed in \cite{Tominaga}. The relations of elements in DSS and Jacobi forms are summarized as follows.

\begin{enumerate}
    \item Elements with positive $y$-coordinates have trivial image in $e_\JF \colon \TEJF_4 \to \TJF_4 \to \JF_4$.
    \item The generators in bidegree $(0,0)$ and $(4,0)$ in the $E_2$-term (Figure \ref{fig:E2TEJF4}) correspond to $a^4 \in \JF_2$ and $a^2b \in \JF_2$, respectively. It turns out that $a^2b$ supports $d_3$-differential, and the class represented by $2b$ survives in the $E_\infty$-page. Therefore the generator of $\pi_4 \TEJF_4$ maps to $2a^2b$ via $e_\JF \colon \pi_* \TEJF_4 \to \JF_4$.
    \item The generator of $E_2^{8,0}$ in Figure \ref{fig:E2TEJF4} represents $d \in \JF_4$. 
    \item In the $E_2$-page (Figure \ref{fig:E2TEJF4}), the classes that are divisible by $c_4$ or $c_6$ are drawn separately above the main part of the chart. These elements exhibit a $\ko$-like pattern and are periodic under multiplication by $c_4$, $c_6$, and $\Delta$. For simplicity, these classes are omitted from the $E_4$-page diagram.
\end{enumerate}

\begin{figure}
    \centering
    \includegraphics[page = 1, width=1\linewidth]{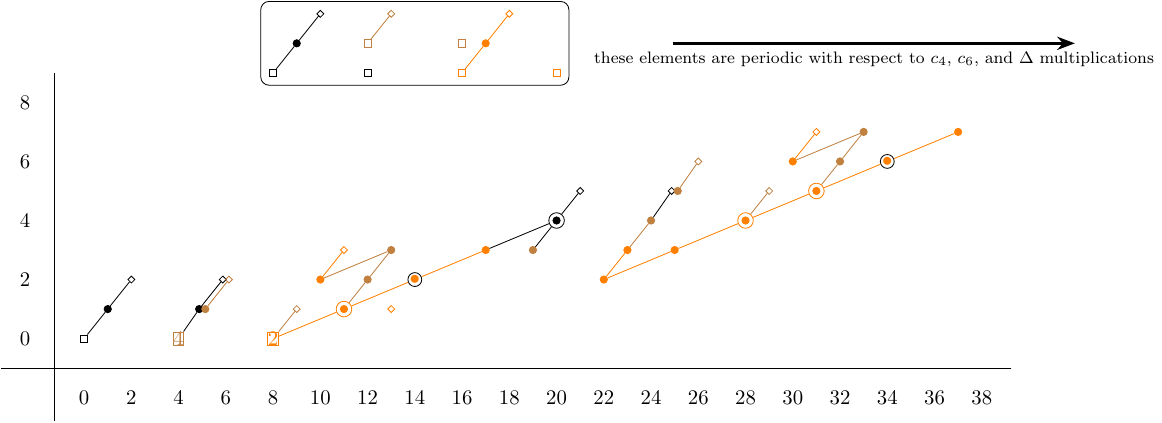}
    \caption{The $E_2$-page of DSS for $\TEJF_4$.}
    \label{fig:E2TEJF4}
\end{figure}

\begin{figure}
    \centering
    \includegraphics[page = 2, width=1\linewidth]{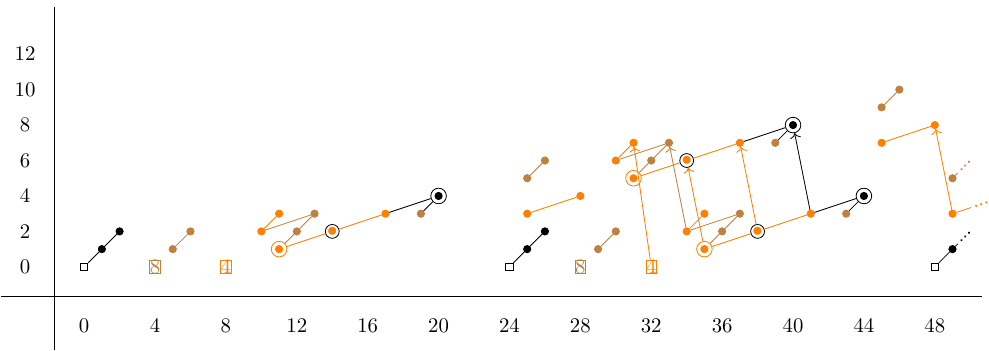}
    \caption{The $E_4$-page of DSS for $\TEJF_4$ and differentials $d_r$, $d \geq 5$.}
    \label{fig:TEJF4diff1}
\end{figure}

\begin{figure}
    \centering
    \includegraphics[page = 3, width=1\linewidth]{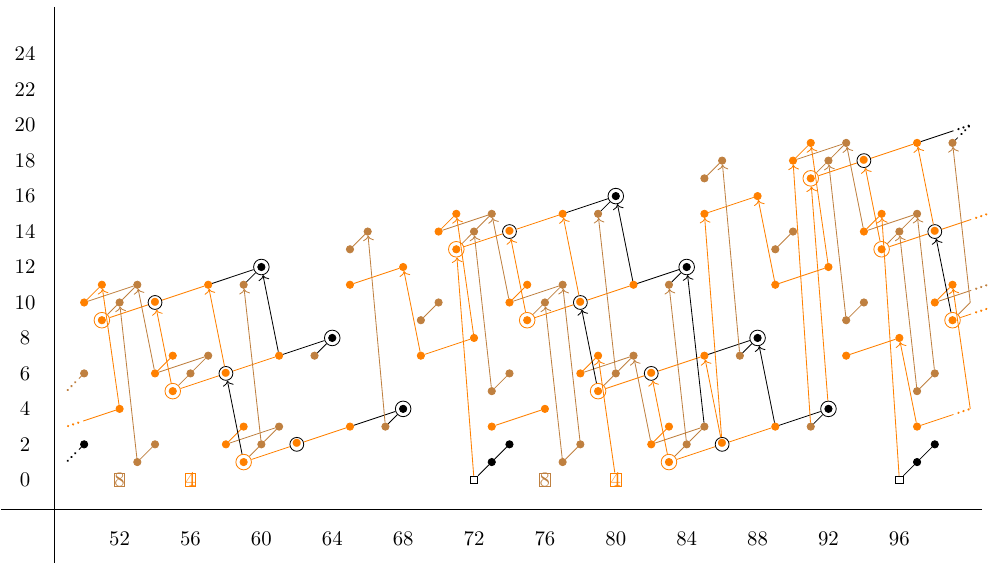}
    \caption{The $E_4$-page of DSS for $\TEJF_4$ and differentials $d_r$, $d \geq 5$.}
    \label{fig:TEJF4diff2}
\end{figure}

\begin{figure}
    \centering
    \includegraphics[page = 4, width=1\linewidth]{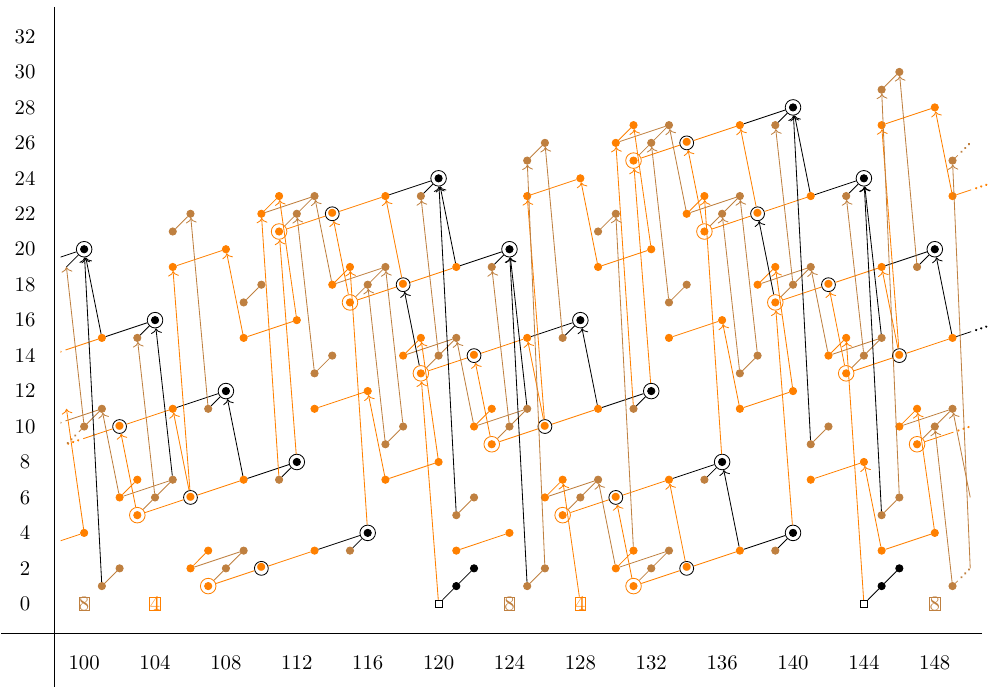}
    \caption{The $E_4$-page of DSS for $\TEJF_4$ and differentials $d_r$, $d \geq 5$.}
    \label{fig:TEJF4diff3}
\end{figure}

\begin{figure}
    \centering
    \includegraphics[page = 5, width=1\linewidth]{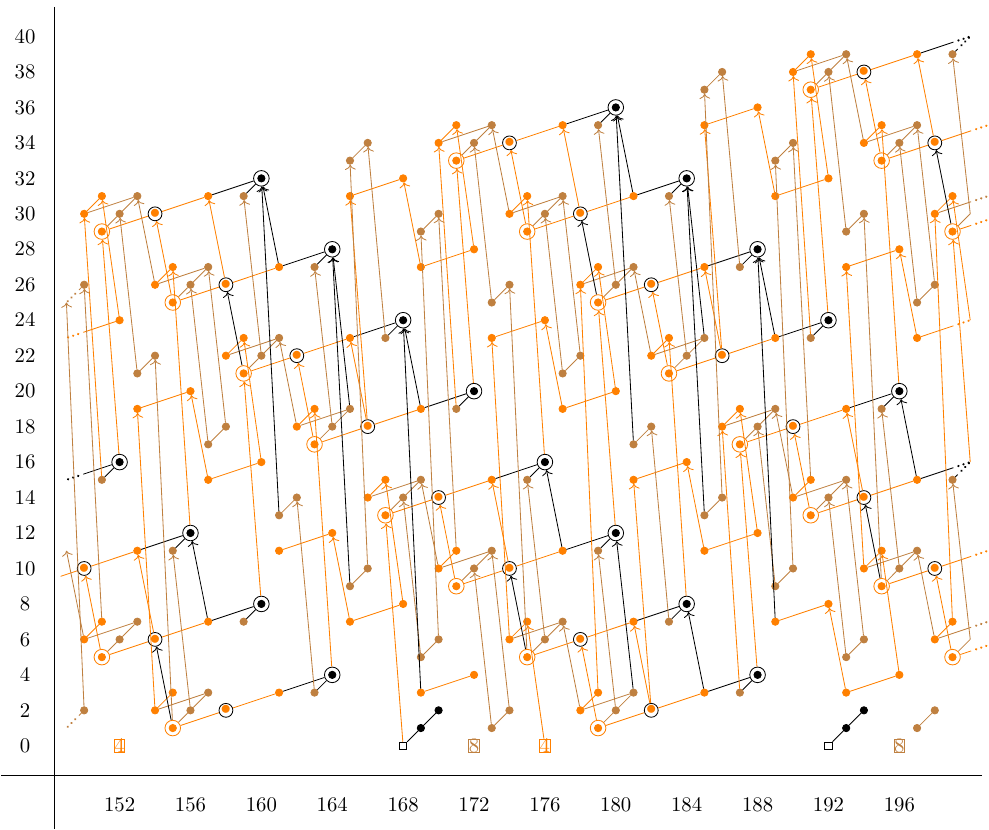}
    \caption{The $E_4$-page of DSS for $\TEJF_4$ and differentials $d_r$, $d \geq 5$.}
    \label{fig:TEJF4diff4}
\end{figure}

\begin{figure}
    \centering
    \includegraphics[page = 6, width=1\linewidth]{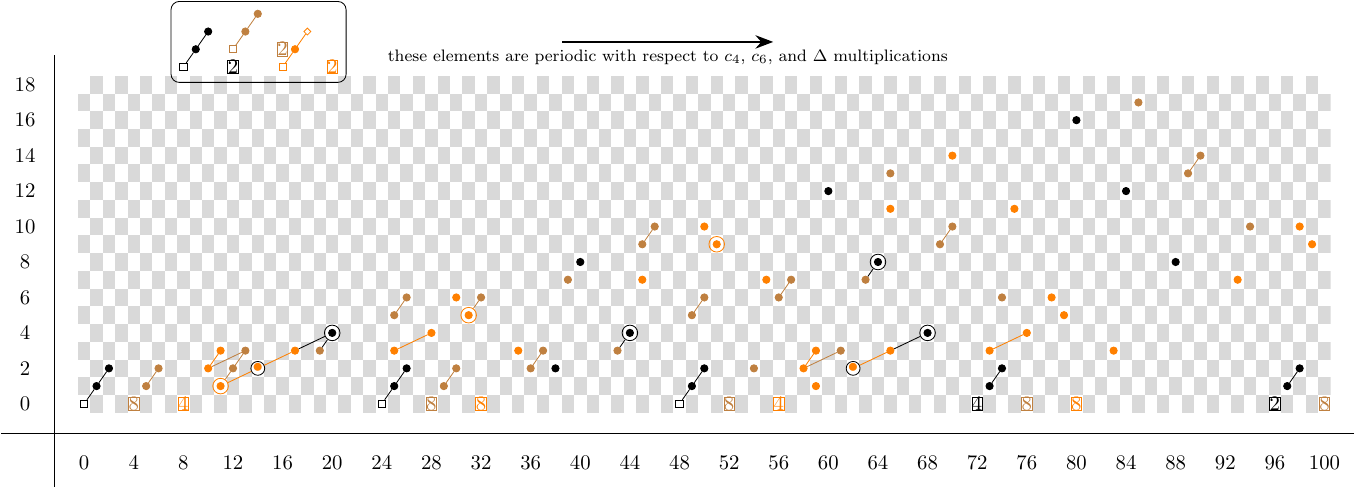}
    \caption{The $E_\infty$-page of DSS for $\TEJF_4$.}
    \label{fig:piTEJF4first}
\end{figure}

\begin{figure}
    \centering
    \includegraphics[page = 7, width=1\linewidth]{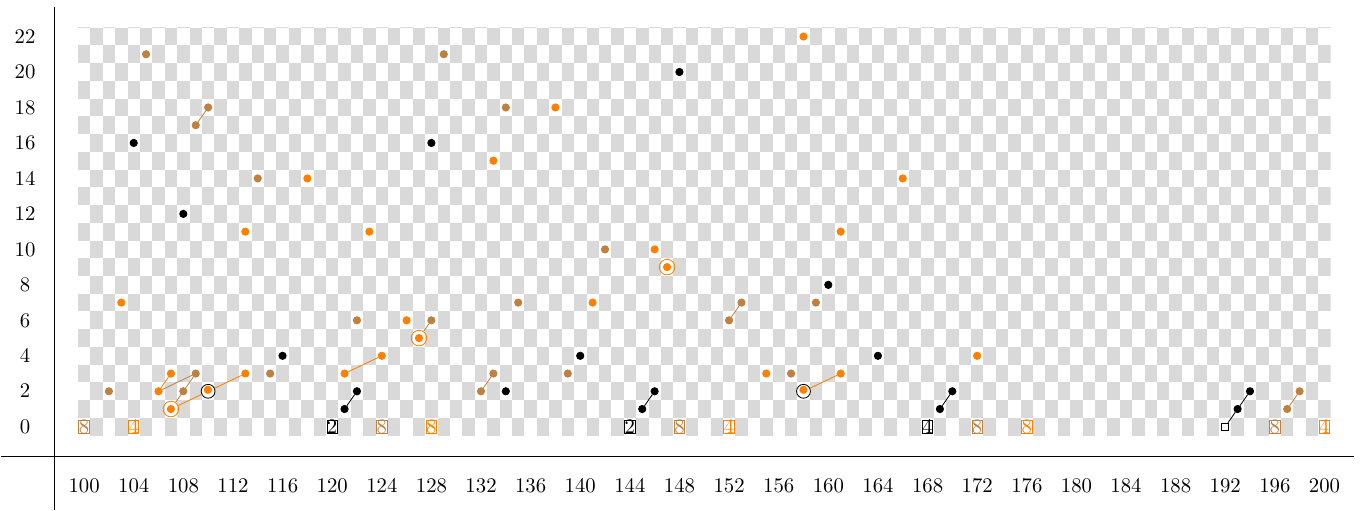}
    \caption{The $E_\infty$-page of DSS for $\TEJF_4$.}
    \label{fig:piTEJF4second}
\end{figure}

\begin{figure}
    \centering
    \includegraphics[page = 7, width=1\linewidth]{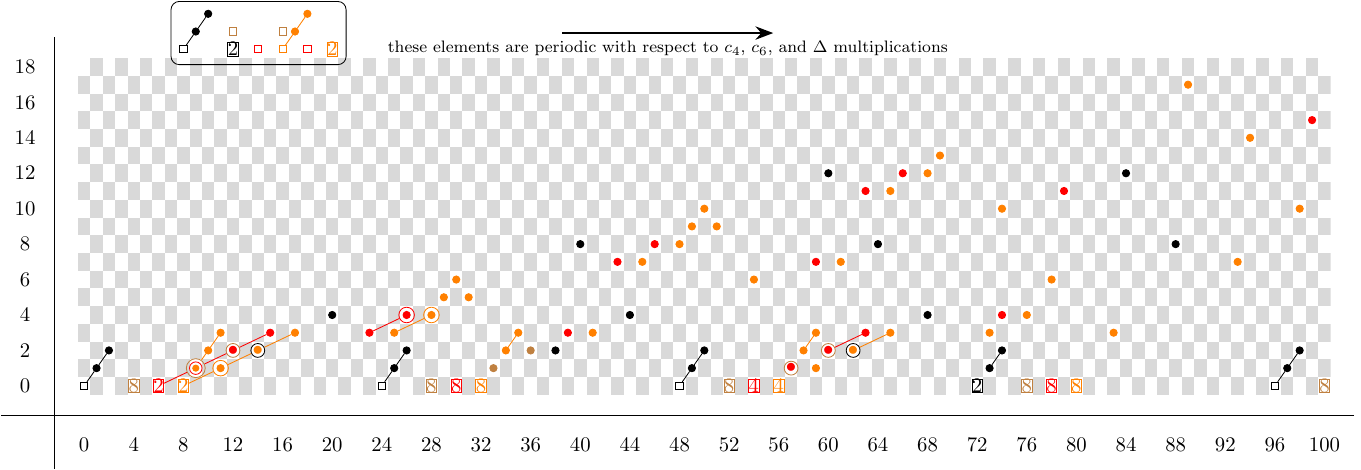}
    \caption{The $E_\infty$-page of DSS for $\TJF_4$.}
    \label{fig:piTJF4first}
\end{figure}

\begin{figure}
    \centering
    \includegraphics[page = 8, width=1\linewidth]{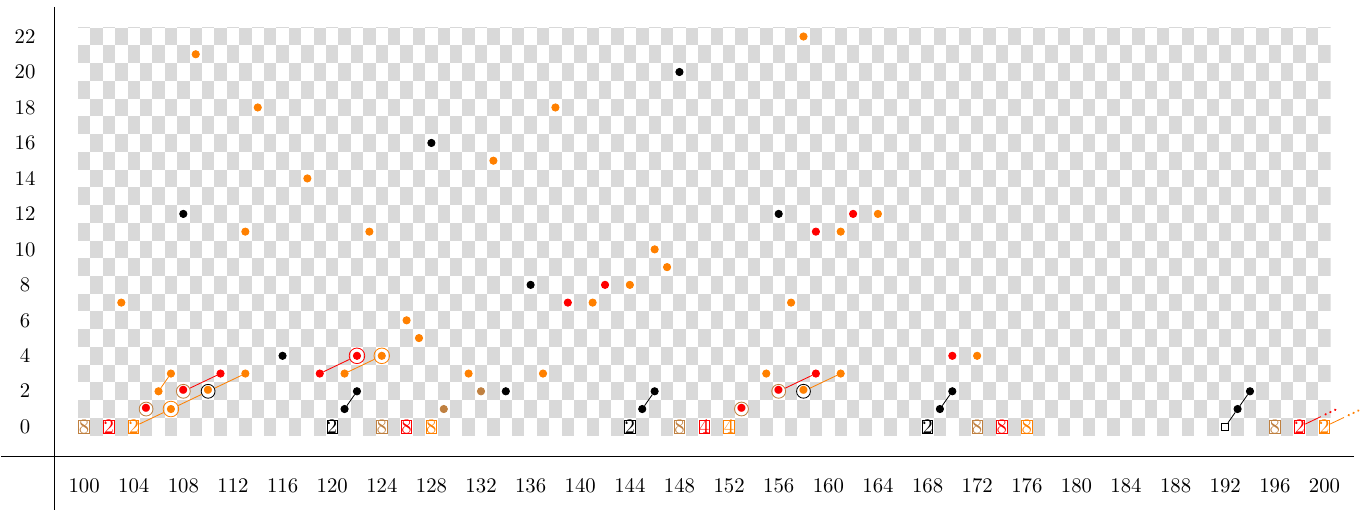}
    \caption{The $E_\infty$-page of DSS for $\TJF_4$.}
    \label{fig:piTJF4second}
\end{figure}

\clearpage

\def\arxivconsists of font{\rm}
\bibliographystyle{ytamsalpha}
\bibliography{ref}

\end{document}